\documentclass[reqno,12pt]{amsart}

\usepackage{geometry,amsmath,amsfonts,amssymb,amsthm,amscd,mathrsfs,graphicx,enumerate,multicol,wrapfig,subfigure,hyperref,stmaryrd, accents,emptypage}
\usepackage[all]{xy}
\usepackage[nodisplayskipstretch]{setspace}
\usepackage{setspace}
\usepackage{tikz-cd}

\usepackage[T1]{fontenc}
\newcommand{\ra}{\rightarrow}

\newcommand{\lra}{\longrightarrow}

\newcommand{\p}{\prime}
\newcommand{\pt}{\partial}

\newcommand{\al}{\alpha}
\newcommand{\Om}{\Omega}
\newcommand{\om}{\omega}
\newcommand{\gam}{\gamma}
\newcommand{\Gam}{\Gamma}
\newcommand{\s}{\sigma}
\newcommand{\vp}{\varphi}

\newcommand{\q}{\theta}

\newcommand{\be}{\beta}
\newcommand{\dt}{\delta}

\newcommand{\Zbb}{\mathbb{Z}}

\newcommand{\Cbb}{\mathbb{C}}

\theoremstyle{plain} 

\newtheorem{THM}{Theorem}[section]

\newtheorem{PROP}[THM]{Proposition}
\newtheorem{LEM}[THM]{Lemma}
\newtheorem{COR}[THM]{Corollary}
\newtheorem{REM}[THM]{Remark}

\newtheorem*{THMMAIN}{Main Theorem}

\newcommand{\bt}{\bullet}

\newcommand{\img}{\mathrm{im}}

\newcommand{\Uc}{\mathcal{U}}

\newcommand{\Oc}{\mathcal{O}}

\newcommand{\Cc}{\mathcal{C}}

\newcommand{\Fc}{\mathcal{F}}

\newcommand{\Abb}{\mathbb{A}}

\newcommand{\Xfr}{\mathfrak{X}}
\newcommand{\Ec}{\mathcal{E}}

\newcommand{\Mfr}{\mathfrak{M}}

\newcommand{\Scl}{\mathcal{S}}
\newcommand{\Xc}{\mathcal{X}}
\newcommand{\Mcl}{\mathcal{M}}

\showoutput
\usepackage{xcolor}
\definecolor{airforceblue}{rgb}{0.36, 0.54, 0.66}
\definecolor{burgundy}{rgb}{0.5, 0.0, 0.13}
\definecolor{majorelleblue}{rgb}{0.38, 0.31, 0.86}
\definecolor{darkblue}{rgb}{0.0, 0.0, 0.55}

\hypersetup{urlcolor=burgundy, linkcolor=burgundy,
citecolor=darkblue,
colorlinks=true}

\title[Genus two supermoduli]
{On the splitting of genus two Supermoduli
\\}
\author{\small Kowshik Bettadapura}
\date{}

\begin{document}

\begin{abstract} 
This article investigates why the genus two, supermoduli space of curves will split in contrast to, potentially, almost all other supermoduli spaces. We use that the dimension of the odd, versal deformation space of a genus two, super Riemann surface is two dimensional. As a consequence, the odd versal deformations can be generated by Schiffer variations at the associated points of a Szeg\"o kernel. This idea is present in D'Hoker and Phong's two loop, superstring amplitude calculation. We show how this idea, combined with Donagi and Witten's characterization of supermoduli obstructions, will result in a splitting of supermoduli space in genus two. 
%14H15, 14H55, 32C11,  58A50
\end{abstract}

\maketitle
\thispagestyle{empty}

\thispagestyle{empty}

\setcounter{tocdepth}{1}
\tableofcontents

%\section*{~}
%\thispagestyle{empty}

\section*{Introduction}
\noindent
The $n$-pointed, supermoduli space of genus $g$ curves, $\Mfr_{g, n}$, is known \emph{not} to split in the following ranges:
\begin{align}
\begin{array}{ll}
g \geq n+1 & \mbox{if $n> 0$}
\\
g \geq 5 &\mbox{if $n = 0$}
\end{array}
\label{rnuierbfuienfoie}
\end{align}
This property was established by Donagi and Witten in \cite{DW1}. Furthermore, it was conjectured $\Mfr_{g, n}$ will not split for $g\geq 3$ and $n = 0$. This remains unresolved and remarks by Witten in \cite{WITTHOLOSTRING,WITTSPRD} strongly suggest non-splitting in the instance $g = 3$, $n = 0$. In genus $g = 0, n = 4$, Giddings in \cite{GPUNCT} constructs a splitting through generalizing the familiar cross ratio. This construction is suggested to split supermoduli space for $g = 0$ and any $n$. In the case $(g, n) = (1, 0)$, supermoduli space is split for dimensional reasons. For $g = 1$ and general $n$, supermoduli space appears to split in light of the classical, one loop calculations by Green and Schwarz in \cite{GREENSCHWARZ}. 

This article is concerned with the case $(g, n) = (2, 0)$. 
\\\\
In genus $g = 2$ and $n = 0$ a splitting is known. It was derived by D'Hoker and Phong in a series of papers, reviewed in \cite{HOKERPHONG1}, and exploited en route their calculation of the superstring amplitude at two loops.\footnote{Note, a splitting in the case $(g, n) = (2, 0)$ is indeed consistent with \eqref{rnuierbfuienfoie}.} This splitting is based on the observations: (a) any genus $g = 2$ super Riemann surface $\Scl$ can be identified with its matrix of periods and (b) there exists a global gauge slice wherein this matrix coincides with the period matrix of the curve $C$ underlying $\Scl$. Schematically, it is a mapping: 
\begin{align}
\Scl \equiv \widehat\Om_\Scl \longmapsto \Om_C \equiv C
\label{rh8hf09j3f3}
\end{align}
where $\widehat\Om_\Scl$ and $\Om_C$ denote the period matrices of $\Scl$ and $C$ respectively. Our aim in this article is to use ideas developed by Donagi and Witten in \cite{DW2} to see why \eqref{rh8hf09j3f3} will lead to a holomorphic splitting of $\Mfr_g|_{g=2}$.\footnote{The mapping \eqref{rh8hf09j3f3} also makes sense in genus $g = 3$. It is not holomorphic however. As argued by Witten in \cite[Appdx., C.3]{WITTSPRD}, this mapping will be singular along the hyperelliptic locus.} By way of caveat, what we will in fact argue is that the \emph{locus of generic, even super Riemann surfaces} in $\Mfr_g$ splits for $g = 2$ rather than $\Mfr_g|_{g = 2}$ itself. This is because, off this locus, period matrices will no longer faithfully represent super Riemann surfaces.\footnote{Here is a subtle distinction therefore between super Riemann surfaces and their classical counterparts.}

In sum, in this article we will endeavor to prove: 

\begin{THMMAIN}
The generic, even component of the (unpointed) supermoduli space of curves is split in genus $g = 2$.
\end{THMMAIN}

\noindent
\emph{Remark}. As an aside, note that the bound in \eqref{rnuierbfuienfoie} is on the marked points. Since the supermoduli space we consider here is not compactified, the marked points will not collide. Consequently, it might well imply the genus $g$, $n$-pointed supermoduli space will split for $n$ sufficiently large. This was conjectured by Donagi and Witten in \cite[Remark 1.5]{DW1}. Our method of proof in this article might suggest a possible route to addressing this conjecture.

\subsection*{Outline and proof sketch}
\noindent
We begin this article with some background on Berezinians on supermanifolds. This is in order to view SRS\footnote{In this article we abbreviate `super Riemann surface' to SRS.} periods as giving a morphism from relative Berezinians on SRS families. After briefly introducing deformation parameters and SRS families in \S\ref{fnckbcubcuiencioe}, we present a construction of the relative Berezinian sheaf in \S\ref{rgf784gf74hf98h48fh04}. D'Hoker and Phong's period matrix formula is then described in Proposition \ref{rfh984hf8hf0903fj93}. This description is not gauge-invariant in a broader sense however and, accordingly, is revisited in Proposition \ref{rfhhf98fj90j3f3f54f35f3}. This result realizes D'Hoker-Phong's formula in terms of Donagi and Witten's SRS deformation gauge pairing in \eqref{rciuehcuiecioe}. 
\\\\
A key ingredient in proving Main Theorem above rests in our adaptation of classical Schiffer variations of a curve to super Riemann surfaces. This is subject of \S\ref{rfg674gf84f9h843f}. A useful characterization of such variations in terms of Delta distributions is given in Lemma \ref{rfg748fg784hf89f30}. In Lemma \ref{rfh89hf983hf0309} we see, in analogy with curves, that Schiffer variations will generate the `odd', versal deformation space of a super Riemann surface. Importantly, in genus $g = 2$, this space is $2$-dimensional. Hence, Schiffer variations at two distinct points will generate the odd, versal deformations. Moreover, as described in \S\ref{rfg78gf874gf7h93}, the Szeg\"o kernel appearing in D'Hoker-Phong's period matrix formula will vanish at two distinct points (Lemma \ref{rfh78hf893hf93j0f93j}). This gives therefore a natural set of points at which to form Schiffer variations.\footnote{c.f., \cite[\S7.2]{HOKERPHONG1}.}

With preliminary material established in \S\ref{ruyrbviurbvbru} and \S\ref{rfg674gf84f9h843f}, we proceed to the proof of Main Theorem in \S\ref{rfgt6gf48fh584}.
 
\subsubsection*{Proof sketch of Main Theorem}
Donagi and Witten in \cite{DW2} construct an exact sequence of invertible sheaves on the spin moduli space\footnote{the reduced part or `body' of supermoduli space is the moduli space of spin curves} and show how its extension class gives the \emph{primary} obstruction to splitting supermoduli space (Theorem \ref{rhf983hf8030fj30}). For dimensional reasons, showing the primary class vanishes in genus $g = 2$ suffices to deduce a splitting (Lemma \ref{rfu3f983hfj390f444}). Hence, to prove the theorem, it suffices to construct a splitting of this sequence of sheaves. At a fixed spin curve (generic, even) we begin by looking at pairings with Schiffer variations at the associated points of the Szeg\"o kernel (Lemma \ref{fbcvyvyubeiucievuev}). This leads to a general statement in Proposition \ref{fjbvhfrbvrbvkjnk} and Proposition \ref{rgf784gf874hf984} for genus $g = 2$, giving rise to an explicit splitting in Corollary \ref{rbfuyvfyuf9h3f784}.\footnote{c.f., Remark \ref{rbfef87f83f783} on the similarity with D'Hoker-Phong's gauge slice independence.} Central to our proof is a famous result by M. Noether on the surjectivity of the infinitesimal period map in genus $g = 2$ giving Proposition \ref{rhf8hf98hf030}. In Lemma \ref{ruerivgiuhrouvjioe} we see how the Szeg\"o kernel exists naturally over the locus of generic, even spin curves. This allows for globalizing the local splitting argument at a generic, even spin curve to split Donagi-Witten's sequence of sheaves over this locus of supermoduli space itself.

\subsubsection*{Generalities}
Throughout this article $k$ will denote the field of complex numbers and families of curves and super Riemann surfaces are assumed to be smooth and proper.

\numberwithin{equation}{section}

\section{Periods and gauge pairings}
\label{ruyrbviurbvbru}

\noindent
In this section we will establish a relation between D'Hoker and Phong's period matrix formula and Donagi and Witten's SRS deformation gauge pairing. 
%Before presenting this relation however, it is necessary to introduce these respective formulae. 
We begin with the period matrix formula.

\subsection{Periods of Berezinians}
The volume form, expressed via differential forms, defines an integration measure on (orientable) smooth, compact  manifolds. While one can also construct differential forms on supermanifolds, the analogue of the volume form is given by the \emph{Berezinain form}. Hence, in order to define periods on super Riemann surfaces, it will be necessary to look at their Berezinian.\footnote{In later sections of this article we will not need any direct reference to the Berezinian. It will help clarify the role played by the versal deformation parameters in D'Hoker-Phong's period matrix formula however, so it is useful to include a brief description here. } 

\subsubsection{The Berezinian}\label{rgf674f873g9f7g379f93}
Briefly, the Berezinian is a homomorphism $Ber : \mathsf{GL}_{m|n}(k) \ra \mathsf{GL}_{1|0}(k)$ and can be expressed as a rational function of determinants. On any supermanifold $\Xfr$ and vector bundle or sheaf of $\Oc_\Xfr$-modules $M$, the \emph{Berezinain} of $M$ is a rank $(1|0)$, $\Oc_\Xfr$-module $Ber_M$. It is constructed as the sheaf with transition functions being the Berezinain of the transition functions of $M$. With $X\subset \Xfr$ the reduced space we have the decomposition over $X$:
\begin{align}
Ber_M|_X
\cong 
\det M_+|_X\otimes \det M_-^*|_X.
\label{fnviurbviurnoivno}
\end{align}
where $M = M_+\oplus M_-$. In the case where $\Xfr$ is split and modeled on $(X, T_{X, -}^*)$, any $\Oc_\Xfr$-module $M$ is isomorphic to $M|_X\otimes \wedge^\bt T^*_{X, -}$. Hence that
\begin{align}
Ber~M
\cong 
(Ber~M|_X)\otimes \wedge^\bt T^*_{X, -}
&&
\mbox{(if $\Xfr$ is split).}
\label{dcnkdjbvjkbrve}
\end{align}
The Berezinain of a supermanifold $\Xfr$ is defined to be that of its cotangent sheaf $T^*_\Xfr$. We refer to the sources \cite{YMAN, QFAS, WITTRS} for more details.

\subsubsection{On super Riemann surfaces}
We assume the reader is at least cursorily familiar with the theory of super Riemann surfaces. For our purposes, a super Riemann surface $\Scl$ is a $(1|1)$-dimensional, complex supermanifold.\footnote{What is referred to as a super Riemann surface (SRS) in the literature, e.g., in \cite{RABCRANE} is what we would refer to as `odd, SRS deformations', c.f., \cite{BETTSRS}.}
It is split and is modeled on a spin curve $(C, K_C^{1/2})$.\footnote{This means $C\subset \Scl$ is the reduced space and $K_C^{1/2}$ is the conormal sheaf to the embedding $C\subset \Scl$.}
The genus of $\Scl$ is the genus of $C$ and $\Scl$ is said to be generic and even if its model $(C, K_C^{1/2})$ is generic and even.\footnote{\label{rhf784gf784hf8}A spin curve is generic and even if $h^0(K_C^{1/2}) = 0$.} The following lemma, for now, points to why we need to ultimately restrict to the locus of generic, even super Riemann surfaces in this article.

\begin{LEM}\label{rfh7fh983hf983}
For any generic, even super Riemann surface $\Scl$ there exists a natural isomorphism
\[
H^0(\Scl, Ber_\Scl) \cong H^0(C, K_C).
\]
where $C\subset \Scl$ is the reduced space.
\end{LEM}

\begin{proof}
With $\Scl$ modeled on $(C, K_C^{1/2})$, recall it is split so that,
\begin{align}
Ber_\Scl \stackrel{\Delta}{=} Ber_{T^*_\Scl} 
&\cong (Ber_{T^*_\Scl}|_C)\otimes \wedge^\bt K_C^{1/2}
&&\mbox{(by \eqref{dcnkdjbvjkbrve})}
\notag
\\
&\cong
(K_C\otimes K_C^{-1/2})
\otimes 
(\Oc_C\oplus K_C^{1/2})
&&
\mbox{(by \eqref{fnviurbviurnoivno})}
\label{rfuhf893jf0j3093}
\\
&\cong 
K_C^{1/2} \oplus K_C
\notag
\end{align}
where $K_C^{-1/2} = (K_C^{1/2})^*$ is the dual. Since $h^0(K_C^{1/2}) = 0$ the lemma follows.\footnote{c.f., f.t. \ref{rhf784gf784hf8}.}
\end{proof}

\noindent 
The isomorphism in Lemma \ref{rfh7fh983hf983} can be made more explicit through the `Berezin integration map', described as follows. If the supermanifold $\Xfr$ is equipped with a projection $\Xfr \ra X$ onto its reduced space $X$, Berezin integration is defined as integration along the fibers of the projection. Denoting this by $_{\Xfr/X}\int$, it is a morphism $_{\Xfr/X}\int : Ber_\Xfr \ra \Om^{top.}_X$ from Berezinian forms on $\Xfr$ to volume forms on $X$. Following Witten in \cite{WITTSMFLD}, let $(z|\q)$ denote local coordinates on $\Xfr$. The Berezinain sheaf is generated over $\Oc_\Xfr$ by the formal symbol $[dz|d\q]$. Berezin integration is then given by
\begin{align}
\sideset{_{\Xfr/X}}{}\int[dz|d\q] 
=
0
&&
\mbox{and}
&&
\sideset{_{\Xfr/X}}{}\int\q[dz|d\q] = dz.
%f(z|\q)[dz|d\q]
%\stackrel{~_\pi\int}{\longmapsto}
%\left\{
%\begin{array}{ll}
%f^{top.}(z)~dz&
%\\
%0 &
%\end{array}
%\right.
\label{rfj89f90jf3k04ff34}
\end{align}
In the proof of Lemma \ref{rfh7fh983hf983} we used that the super Riemann surface $\Scl$ is split and in \eqref{rfuhf893jf0j3093} an explicit splitting. Any splitting defines a projection $\Scl \ra C$ and therefore Berezin integration $_{\Xfr/X}\int$. By \eqref{rfj89f90jf3k04ff34}, Berezin integration gives a morphism on global sections $_{\Xfr/X}\int: H^0(Ber_\Scl) \ra H^0(K_C)$. The inclusion $K_C \ra Ber_\Scl$ given in coordinates $(z|\q)$ by $\om(z) \mapsto \q\om(z)$ will be inverse to $_{\Xfr/X}\int$.

%\begin{REM}
%\emph{In the case $\Scl$ is \emph{not} generic, there will be more global Berezinians on $\Scl$ than global differentials on $C$.}
%\end{REM}

%\subsection{D'Hoker and Phong's formula}
\subsection{Periods in families}

\subsubsection{SRS Periods}
On a curve $C$ of genus $g$, its space of global differentials $H^0(K_C)$ is $g$-dimensional. With the Dolbeault isomorphism and the inclusion $H^0(K_C)\cong H_{\overline \pt}^{1, 0}(C)\subset H_{dR}^1(C, \Cbb)$, the familiar pairing between $1$-cycles $\gam$ on $C$ and de Rham classes give the pairing $H_1(C, \Zbb)\otimes H^0(K_C) \ra \Cbb$, $(\gam, \om) \mapsto \oint_\gam\om$. Fixing a transversely intersecting homology basis $(\al_i, \be_i)_{i = 1. \ldots, g}$ on $C$ and a basis of $\al$-normalised, global differentials $(\om_\ell)_\ell$, let $\Om_C$ denote the period matrix of $C$. It is defined through pairing with $\be$-cycles. Its $(i, \ell)$-th entry is
\begin{align}
\Om_{C, ij} = \oint_{\be_i}\om_\ell.
\label{rfg784gfhf89h89f}
\end{align}
Now let $K_C^{1/2}$ be a generic, even spin structure on $C$. Then $(C, K_C^{1/2})$ models a generic, even super Riemann surface $\Scl$. By Lemma \ref{rfh7fh983hf983}, any basis of global differentials $(\om_\ell)_\ell$ on $C$ can be placed in bijective correspondence with global Berezinians $(\s_\ell)_\ell$ on $\Scl$. The periods of these Berezinians are defined analogously to \eqref{rfg784gfhf89h89f}. With respect to the homology basis $(\al_i, \be_i)_i$ and a projection $\Scl \ra C$ the period matrix of $\Scl$, denoted $\widehat\Om_\Scl$, is the matrix with $(i, \ell)$-th entry:
\begin{align}
\widehat\Om_{\Scl, i\ell} \stackrel{\Delta}{=} \oint_{\be_i}\sideset{_{\Scl/C}}{}\int \s_\ell.
\label{iuriurirnoiifmeppe}
\end{align}
With our conventions we have, thus far, not deduced anything interesting. Indeed, by \eqref{rfj89f90jf3k04ff34} see that $_{\Scl/C}\int\s_\ell = \om_\ell$ giving therefore $\widehat\Om_{\Scl} = \Om_{C}$. Periods of \emph{relative} Berezinians are however a more interesting matter. It is here where we will understand D'Hoker and Phong's formula.

\subsubsection{Relative periods on curves}
The period matrix construction on a curve allows for defining a morphism on the moduli space of curves. That is, following \cite[Ch. XI, \S8]{ARABGRIFF}, let $X/B$ be a family of curves with fixed topological type. This means we are free to fix a homology basis $(\al_i, \be_i)$ on $X$ which restricts to a homology basis of $1$-cycles on each fiber $C_b\subset X$. Fixing this, we can form a map $b \mapsto \oint_{\be_i} w_\ell(b)$ where $w_\ell(b)$ is a family of global differentials on $X$, i.e., $w_\ell(b)\in H^0(C_b, K_{C_b})$. Phrased alternatively, if $\pi : X\ra B$ denotes the family of curves, $\om = (w_\ell(b))_{b\in B}$ defines a relative differential $\pi_*\om_\ell\in \pi_*\Om^1_{X/B}$ where $\Om^1_{X/B}$ is the relative cotangent sheaf on $X$.
%\footnote{c.f., \cite[Ch. III, \S9]{HARTALG}.} 
Then as in \eqref{rfg784gfhf89h89f} we can define the relative matrix of periods $\Om_{X/B}$ with $(i, \ell)$-th entry,
\begin{align}
\Om_{X/B, i\ell} = \oint_{\be_i}\pi_*\om_\ell.
\label{rf783gf7hf98h38f3}
\end{align}
The row space of $\Om_{X/B}$ is, for each $b\in B$, a $(h^{1, 0}_{\overline \pt}(C) = g)$-dimensional subspace of $H^1(C, \Cbb)$ and changing the basis $(\om_\ell)_\ell$ gives a linear isomorphism of this subspace. Hence, as described in \cite[p. 218]{ARABGRIFF}, we obtain a mapping into the Grassmannian, $B \ra \mathbb Gr(g, H^1(C, \Cbb))$. Note that it will necessarily factor through the moduli space $\Mcl_g$. Its differential is related to cupping with the Kodaira-Spencer class of the family. 

Our interest lies in the expression \eqref{rf783gf7hf98h38f3} adapted to super Riemann surfaces.

\subsubsection{SRS families and Kodaira-Spencer}\label{fnckbcubcuiencioe}
Families of super Riemann surfaces come in two distinct kinds: \emph{even} and \emph{odd}, owing the two parities abound in supergeometry. The even deformations are classified as in the case of curves. The odd deformations are classified by the spin structure. If $\Scl$ is a super Riemann surface modeled on the spin curve $(C, K_C^{1/2})$, its `odd', versal deformation space is $H^1(C, K_C^{-1/2})$. With $\Mfr_g$ the supermoduli space of curves then, any odd point $\Abb_k^{0|q} \ra \Mfr_g$ corresponds to a family of super Riemann surfaces $\Xc \stackrel{\pi}{\ra} \Abb^{0|q}_k$. In analogy with the case of curves, the differential at $\xi =0$ gives an isomorphism $d\pi_0: T_{\Abb^{0|q}_k, 0} \stackrel{\sim}{\ra}  ({\bf R}^1\pi_*T_{\Xc/\Abb^{0|q}_k})_0 \cong \oplus^qH^1(C, K_C^{-1/2})$. This is argued in \cite{BETTSRS}. The Kodaira-Spencer class is then 
\begin{align}
\kappa_{\Xc, 0} 
= 
d\pi\left(\frac{\pt}{\pt \xi_1}\big|_{\xi=0}, \ldots, \frac{\pt}{\pt \xi_q}\big|_{\xi=0}\right).
\label{rfg73f983h8f03}
\end{align}
We refer to \cite{BETTPHD, BETTSRS} for a more detailed study into the deformation theory for super Riemann surfaces. For the purposes of this article it suffices to know that $\kappa_{\Xc, 0}$ can also be represented in smooth coordinates $(z, \overline z|\q, \xi)$ on $\Xc$ by the Berezinian form,\footnote{See \cite{BETTANAL}.}
\begin{align}
\kappa_{\Xc, 0}
&\equiv 
\left[
\sum_m \xi_m\chi^m(z, \overline z)\q~[dzd\overline z|d\q]
\right]
\label{rfh89hf983fh093j0}\\
&=
\sum_m
\xi_m \big[\chi^m(z, \overline z)\q~[dzd\overline z|d\q]
\big]
\label{rfiuhf983hf0j390}
&&
\mbox{(by \eqref{rfg73f983h8f03})}
\end{align}
where the equivalence $[-]$ in \eqref{rfh89hf983fh093j0} is that defining cohomology classes in $H^1(K_C^{-1/2}) \cong A^{0,1}(K_C^{-1/2})/\overline \pt A^{0, 0}(K_C^{-1/2})$. This description can be arrived at through studying perturbations of the Dolbeault operator on $\Scl$. In the context of Berezinian forms, define the operator $\overline\pt_{Ber}  = [d\overline z]\frac{\pt}{\pt \overline z}$. Setting $\chi = \sum_m \chi^m(z, \overline z)\q~[dzd\overline z|d\q]$, see that $\chi$ will be a smooth form representative for the Kodaira-Spencer class. As a gauge field it transforms via the Berezinian form $u = \sum_m \xi_m u^m(z, \overline z)\q~[dz|d\q]$ by
\begin{align}
\chi
\longmapsto&~
\chi + \overline \pt_{Ber} \sum_m \xi_mu^m(z, \overline z)\q~[dz|d\q]
\notag
\\
=&~
\chi + \sum_m \xi_m\frac{\pt u^m(z, \overline z)}{\pt \overline z}\q~[dzd\overline z|d\q]
\notag
\\
=&~
\chi + \overline \pt_{Ber}u.
\label{rfh983hf983f09j93j}
\end{align}
We will see later how \eqref{rfh983hf983f09j93j} compares more generally with Donagi and Witten's gauge equivalence. For now, note that two descriptions have been given for deformations of super Riemann surfaces: an `algebraic' description in \eqref{rfg73f983h8f03} and an `analytic' description in \eqref{rfh89hf983fh093j0}. The subject of \cite{BETTANAL} is to show the two descriptions are compatible in the following sense:

\begin{THM}\label{rfhhf893f0j39f03}
SRS families over $\Abb^{0|q}_k$ with central fiber $\Scl$ up to isomorphism are in bijective correspondence with perturbations of the Dolbeault operator on $\Scl$ over $\Abb^{0|q}_k$ up to gauge equivalence.
\qed
\end{THM}

%\noindent
%The gauge equivalence mentioned in Theorem \ref{rfhhf893f0j39f03} above is precisely Donagi and Witten's gauge 
%In \eqref{rfiuhf983hf0j390}, the Kodaira-Spencer class of $\Xc\stackrel{\pi}{\ra}\Abb^{0|n}_k$ is expressed in terms of versal deformation parameters.

\begin{REM}\label{rfh983hf893fj390}
\emph{To avoid possible confusion, it is worth mentioning explicitly that the Kodaira-Spencer class will only classify SRS families over $\Abb^{0|q}_k|_{q = 1}$. For general $q> 1$ there will be deformation parameters proportional to $\xi_m\xi_n$ and $\xi_m\xi_n\xi_o$ and so on. The transformations in \eqref{rfh983hf983f09j93j} are, in this sense, the infinitesimal part of the gauge equivalence mentioned in Theorem \ref{rfhhf893f0j39f03}. The family will nevertheless define a Kodaira-Spencer class as in \eqref{rfg73f983h8f03}. There will however exist, for general $q> 1$, inequivalent families with  coincident Kodaira-Spencer classes. See e.g., \cite[\S3.4, Lemma 3.15]{BETTSRS}.}
\end{REM}

\subsubsection{The relative Berezinian sheaf}\label{rgf784gf74hf98h48fh04}
Recall that the Berezinain sheaf in supergeometry is the analogue of the sheaf of holomorphic volume forms in algebraic geometry. This analogue extends to the case of families where, like the relative canonical bundle, we can define the notion of `relative Berezinian sheaf' on a family  $\Xc \stackrel{\pi}{\ra} \Abb^{0|q}_k$, denoted $Ber_{\Xc/\Abb_k^{0|q}}$. Indeed, a hint toward forming such a sheaf appears in Witten's derivation of D'Hoker and Phong's period formula in \cite[\S8]{WITTRS}.\footnote{A more formal, albeit abstract, argument for its existence might also be seen through Penkov's characterisation of the Berezinian sheaf as the dualizing object in supergeometry in \cite{PENKOV}.} We give a construction of the relative Berezinian sheaf in what follows.
\\\\
Starting with our family $\Xc \stackrel{\pi}{\ra}\Abb_k^{0|n}$ of super Riemann surfaces, we have a short exact sequence of cotangent sheaves $0 \ra \pi^*T_{\Abb^{0|q}_k}^*\ra T^*_\Xc \ra T^*_{\Xc/\Abb^{0|q}_k}\ra 0$.\footnote{Note, this sequence is short exact since the map $\pi$ is assumed to be smooth (see e.g., \cite[Ch. II, \S8]{HARTALG}).} Like the determinant, the Berezinian is multiplicative giving thereby an isomorphism $Ber_\Xc \cong \pi^*Ber_{\Abb^{0|q}_k}\otimes Ber~T^*_{\Xc/\Abb^{0|q}_k}$. In local coordinates $(z|\q, \xi)$ on $\Xc$, this isomorphism reflects the identification $[dz|d\q d\xi] = [d\xi]\otimes [dz|d\q]$, where $(z|\q)$ are coordinates on the fiber super Riemann surface and $\xi$ are base coordinates. We now define the relative Berezinian sheaf thusly, 
\[
Ber_{\Xc/\Abb^{0|q}_k} \stackrel{\Delta}{=} Ber~T^*_{\Xc/\Abb^{0|q}_k}.
\]
%As $\Abb^{0|q}_k$ is canonically split, it projects onto its reduced space $\mathbb A^{0|q}_k \ra \mathrm{Spec}~k$.\footnote{With $\Abb^{0|q}_k = \mathrm{Spec}~k[\xi_1, \ldots, \xi_q]$, the projection $\Abb^{0|q}_k \ra \mathrm{Spec}~k$ reflects that $k[\xi_1, \ldots, \xi_q]$ is a $k$-algebra.} This gives a morphism $_{\Abb^{0|q}_k/k}\int : Ber_\Xc \ra Ber~T^*_{\Xc/\Abb^{0|q}_k}$ integrating our the base parameters. We now define the relative Berezinian sheaf thusly, 
%\[
%Ber_{\Xc/\Abb^{0|q}_k} \stackrel{\Delta}{=} \frac{Ber~T^*_{\Xc/\Abb^{0|q}_k}}{\img\left\{ _{\Abb^{0|n}_k/k}\int: Ber_\Xc \ra \pi_*Ber~T^*_{\Xc/\Abb^{0|n}_k}\right\}}.
%\]
%By the projection formula, $\pi_*Ber_\Xc \cong Ber_{\Abb^{0|q}_k}\otimes \pi_*Ber~T^*_{\Xc/\Abb^{0|q}_k}$.
%\newpage
%Through Berezin integration we obtain a morphism of sheaves $_{\Abb^{0|q}_k/k}\int\otimes 1 : Ber_{\Abb^{0|q}_k} \otimes \pi_* Ber~T^*_{\Xc/\Abb^{0|q}_k} \ra \pi_*Ber~T^*_{\Xc/\Abb^{0|q}_k}$ integrating out the base parameters. The \emph{relative Berezinian sheaf} is now defined to be the subsheaf $Ber_{\Xc/\Abb^{0|q}_k}\subset Ber~T^*_{\Xc/\Abb^{0|q}_k}$ such that along $\pi$,
%\[
%\pi_*Ber_{\Xc/\Abb^{0|n}}
%\cong
%\frac{\pi_*Ber~T^*_{\Xc/\Abb^{0|n}_k}}{\img\left\{ _{\Abb^{0|n}_k/k}\int\otimes 1: Ber_{\Abb^{0|n}_k} \otimes \pi_* Ber~T^*_{\Xc/\Abb^{0|n}_k} \ra \pi_*Ber~T^*_{\Xc/\Abb^{0|n}_k}\right\}}.
%\]
To compare with the Berezinian forms considered by Witten in \cite[\S8]{WITTRS}: in the coordinates $(z|\q, \xi)$, local sections $\s\in Ber_{\Xc/\Abb^{0|q}_k}$ can be expressed,
\begin{align}
\s(z|\q,\xi) 
&= f(z|\q, \xi)[dz|d\q]
\notag
\\
&= \big(\al(z) + \q b(z) + \xi (\eta(z) + \q h(z))\big)[dz|d\q].
\label{rf4f4hf98j39fj903j0}
\end{align} 
Recall from \S\ref{rgf674f873g9f7g379f93} that the Berezinian will have rank $(1|0)$. As such the relative Berezinian also has rank $(1|0)$ over $\Abb^{0|q}_k$. The measure $[dz|d\q]$ itself, being the Berezinian measure on the fiber $\Scl\subset \Xc$ over $\xi = 0$, has odd parity. As $\xi$ and $\q$ are also odd, it follows that $\al, \eta$ must be even while $b, h$ are odd. The following is now immediate.

\begin{LEM}\label{rgf974fh893h8fh309}
Let $\Xc \stackrel{\pi}{\ra}\Abb^{0|q}_k$ be a family of generic, even super Riemann surfaces with $\Scl = \pi^{-1}(0)$. Then there exists an isomorphism $(\pi_*Ber_{\Xc/\Abb^{0|q}_k})_{\xi = 0} \cong H^0(C, K_C)$ where $C$ is the reduced space of $\Scl$.
\qed
\end{LEM}

\subsubsection{Relative periods of Berezinians}
Sections of the relative Berezinian capture the intuition of a family of Berezinian forms on a super Riemann surface varying with respect to parameters on the base $\Abb^{0|q}_k$. Now with $\Scl\subset \Xc \ra \Abb^{0|q}_k$ the fiber over $\xi = 0$, Berezin integration $_{\Scl/C}\int : Ber_\Scl \ra K_C$ can be extended to sections of $\pi_*Ber_{\Xc/\Abb^{0|n}_k}$, as can be seen from the coordinate expression in \eqref{rf4f4hf98j39fj903j0}. Furthermore, any $1$-cycle $\gam$ on $C$ can be paired with $\s\in Ber_{\Xc/\Abb^{0|q}_k}$ through pairing with $_{\Scl/C}\int\pi_*\s$. This results in a $k[\xi]$-valued pairing. Consequently we can, as in \eqref{rf783gf7hf98h38f3}, form the relative period morphism as a function of the base parameters $\xi$. That is, for $\Xc\stackrel{\pi}{\ra} \Abb^{0|q}_k$ now a family of \emph{generic, even super Riemann surfaces}, let $(\s_\ell)_\ell$ be a basis of relative Berezinian sections and denote the relative period morphism by $\widehat\Om_{\Xc/\Abb^{0|q}_k}$. It is defined as the morphism with $(i, \ell)$-th component,
\begin{align}
\widehat\Om_{\Xc/\Abb^{0|q}_k, i\ell}= \oint_{\be_i} \sideset{_{\Scl/C}}{}\int \pi_*\s_\ell.
\label{rfu3hf98j30f9j334ff3}
\end{align}
That $\widehat\Om_{\Xc/\Abb^{0|n}_k}$ will be valued in $k[\xi]$ is implied by the term $\pi_*\s_\ell$ (c.f., \eqref{rf4f4hf98j39fj903j0}).

\subsection{D'Hoker and Phong's formula}

\subsubsection{The quadratic component}
The morphism $\widehat\Om_{\Xc/\Abb^{0|q}_k}$ is valued in $k[\xi]$ and at $\xi = 0$ we have from Lemma  \ref{rgf974fh893h8fh309}, $_{\Scl/C}\int\pi_*\s\big|_{\xi = 0}\in H^0(K_C)$. In the case where $\Xc\stackrel{\pi}{\ra} \Abb^{0|q}_k$ is a family of generic, even super Riemann surfaces, a basis $(\pi_*\s_\ell)_\ell$ corresponds to a basis of global differentials on $C$ by Lemma \ref{rfh7fh983hf983}. From \eqref{rf4f4hf98j39fj903j0} and Lemma \ref{rgf974fh893h8fh309} we see evidently,
\begin{align}
\widehat\Om_{\Xc/\Abb^{0|q}_k} = \Om_C + \cdots
\label{Prfh89fh830f09j30}
\end{align}
where $\Om_C$ is the period matrix of $C$. D'Hoker and Phong in \cite{CHIRALSPLT, HOKDEFCPLX} derive a recursive formula filling in the ellipses in \eqref{Prfh89fh830f09j30} for any base $\Abb^{0|q}_k$. Nilpotence of the base parameters imply the ellipses in \eqref{Prfh89fh830f09j30} consist of a finite number terms. More specifically, since $\pi_*Ber_{\Xc/\Abb^{0|q}_k}$ has rank $(1|0)$ over $\Abb^{0|q}_k$, the period morphism $\widehat\Om_{\Xc/\Abb^{0|q}_k}$ will be valued in the even subalgebra $k[\xi]^+\subset k[\xi]$. For $\xi = (\xi_1, \ldots, \xi_n)$ we can continue \eqref{Prfh89fh830f09j30} by,
\begin{align}
\widehat\Om_{\Xc/\Abb^{0|q}_k}
=
\Om_C 
+
\sum_{m < n} 
\widehat \Om_{\Xc/\Abb^{0|q}_k}^{mn}~\xi_m\xi_n
+
\sum_{m< n< o< p}
\widehat \Om^{mnop}_{\Xc/\Abb^{0|q}_k}
~\xi_m\xi_n\xi_o\xi_p
+
\cdots
\label{rfh983hf83f09j3334}
\end{align}
Our interest lies in the component $\widehat\Om_{\Xc/\Abb^{0|q}_k}^{(2)} \stackrel{\Delta}{=} \widehat\Om_{\Xc/\Abb^{0|q}_k} - \Om_C \mod (\xi^3)$. By \eqref{rfh983hf83f09j3334} see that we can write
\[
\widehat\Om_{\Xc/\Abb^{0|q}_k}^{(2)}
=
\sum_{m< n} \xi_m\xi_n\left.\frac{\pt^2}{\pt \xi_n\pt \xi_m} \widehat\Om_{\Xc/\Abb^{0|q}_k}\right|_{\xi = 0}.
\]
We will refer to  $\widehat\Om_{\Xc/\Abb^{0|n}_k}^{(2)}$ as the \emph{quadratic component} of the period morphism $\widehat\Om_{\Xc/\Abb^{0|q}_k}$.

D'Hoker and Phong show the following:\footnote{Note, this statement is rephrased in the language used in this article.}

\begin{PROP}\label{rfh984hf8hf0903fj93}
Let $(\Scl \subset \Xc\ra\Abb^{0|q}_k)$ be a family of generic, even super Riemann surfaces with $\Scl$ modeled on $(C, K_C^{1/2})$. Let $(\om_\ell)_\ell$ be a basis of global differentials on the underlying curve $C\subset \Scl$ defining the period matrix $\Om_C$ and $\chi^m(z, \overline z)d\overline z$ the odd, versal deformation parameters classifying the family. Then the quadratic component of $\widehat\Om_{\Xc/\Abb^{0|q}_k}$ is given by the $(i,\ell)^{(m, n)}$-th entry,
\begin{align}
\widehat\Om_{\Xc/\Abb^{0|q}_k, i\ell}^{mn}
&=
\left.\frac{\pt^2}{\pt \xi_n\pt \xi_m} \widehat\Om_{\Xc/\Abb^{0|q}_k, i\ell}\right|_{\xi = 0}
\notag
\\
&=
\int_{C\times C}
\om_i(x)dx
~\chi^m(x, \overline x)d\overline x
~S(x, y)~
\chi^n(y, \overline y)d\overline y
~\om_\ell(y)dy
\label{rfg47gf74hf8f3}
\end{align}
where $(x, y)$ are coordinates on $C\times C$ lying over the coordinate $z\in C$ and $S(x, y)$ is the Szeg\"o kernel function.\qed
\end{PROP}

\subsection{Donagi and Witten's gauge pairing}\label{rbfy4gf78g98fh380fh03}
Most of the material in this section can be found in \cite[\S3]{DW2}. 
%Proofs are supplied by the author.

\subsubsection{Divisors on $C\times C$}
Fix a spin curve $(C, K_C^{1/2})$. As a spin structure $K_C^{1/2}$ is fixed it makes sense to form the half-powers $K_C^{a/2} \stackrel{\Delta}{=}  (K_C^{1/2})^{\otimes a}$ for any integer $a$. Over the product $C\times C$ we can define the sheaf
\begin{align}
\Oc_{C\times C}(a, b, c)
\stackrel{\Delta}{=}
p_1^*K_C^{a/2} \otimes p_2^*K_C^{b/2}\otimes \Oc_{C\times C}(c\Delta),
\label{rf7fh398fh3f03}
\end{align}
where $p_i$ are the factor projection maps $C\times C\rightrightarrows C$; $\Delta : C\subset C\times C$ is the diagonal embedding and $\Oc_{C\times C}(c\Delta)$ is the sheaf of holomorphic functions on $C\times C$ which are meromorphic along the diagonal with pole order at most $c$. Any global section $\vp\in H^0(C\times C, \Oc_{C\times C}(a, b, c))$ can be represented, in local coordinates $(x, y)$ near the diagonal, as
\begin{align}
\vp(x, y)
=
\frac{f(x, y)}{(x - y)^c}
dx^{\otimes a/2}
dy^{\otimes b/2}
%\left(
%\frac{f_c(z, z^\p)}{(z - z^\p)^c}
%+
%\frac{f_{c-1}(z, z^\p)}{(z - z^\p)^{c - 1}}
%+
%\cdots 
%\right)
%dz^{\otimes a/2}
%dz^{\p\otimes b/2}.
\label{rf9hf983h0390fj3}
\end{align}
for $f$ holomorphic. Evidently we have $\Oc_{C\times C}(a, b, c-1)\subset \Oc_{C\times C}(a, b, c)$ given by $\frac{f(x, y)}{(x - y)^{c - 1}} dx^{\otimes a/2}dy^{\otimes b/2} \mapsto \frac{f(x, y)(x - y)}{(x - y)^c}dx^{\otimes a/2}dy^{\otimes b/2}$. From the relation between cotangent sheaves and diagonal morphisms, the restriction-to-diagonal map gives a morphism $\Oc_{C\times C}(a, b, c) \stackrel{res.}{\ra} \Delta_*K_C^{\otimes (a + b - 2c)/2}$ with kernel $\Oc_{C\times C}(a, b, c-1)$. Hence we arrive at a long exact sequence on cohomology,
\begin{align} 
0 
\lra 
H^0(\Oc_{C\times C}(a, b, c-1))
\lra
H^0(\Oc_{C\times C}(a, b, c))
\stackrel{res.}{\lra}&~
H^0(K_C^{\otimes(a + b - 2c)/2})
\label{rf748fh4f0j93}
\\
\lra&~
H^1(\Oc_{C\times C}(a, b, c-1))\lra\cdots
\notag
\end{align}
In a special case we have the following.

\begin{LEM}\label{fh98hf89hf030f3}
For $c = 1$; $a, b, >2$ and genus $g> 1$, $H^1(\Oc_{C\times C}(a, b, c-1)) = (0)$.
\end{LEM}

\begin{proof}
From the K\"unneth formula for sheaf cohomology note that,
\begin{align*}
H^1(C\times C,~ &\Oc_{C\times C}(a, b, 0))
%&=
%H^1(C\times C, K_C^{a/2}\boxtimes K_C^{b/2}) 
\\
&\cong 
\left(H^0(C, K_C^{a/2})\otimes H^1(C, K_C^{b/2})\right)
\oplus
\left(
H^1(C, K_C^{a/2})\otimes H^0(C, K_C^{b/2})
\right).
\end{align*}
By Serre duality, $H^1(C, K_C^{j/2}) \cong H^0(C, K_C^{1 - j/2})^*$. Provided the genus $g(C)> 1$, then for any $j > 2$, $K_C^{1-j/2}$ is a line bundle on $C$ with negative degree. Hence it has no global sections and so $H^1(C, K_C^{j/2}) \cong H^0(C, K_C^{1 - j/2})^*\cong H^0(C, K_C^{1 - j/2}) = (0)$. Thus for any integers $a, b$ with $a>2, b > 2$, that $H^1(C\times C, \Oc_{C\times C}(a, b, 0)) = (0)$. 
\end{proof}

\noindent 
Under the conditions stipulated in Lemma \ref{fh98hf89hf030f3}, we obtain a short exact sequence on global sections:
\begin{align}
0
\lra
H^0(\Oc_{C\times C} (a, b, 0)) 
\lra 
H^0(\Oc_{C\times C} (a, b, 1))
\stackrel{res.}{\lra}
H^0(K_C^{(a+b - 2)/2})
\lra 0.
\label{rfh793hf983hf083}
\end{align}
We turn now to the case $a = b$ and $c = 1$ in what follows.

\subsubsection{The product involution}
On the product $C\times C$ the involution $\iota: (x, y)\mapsto (y, x)$ serves to split the sequence \eqref{rfh793hf983hf083} into a trivial, odd sequence where sections $\vp\in H^0(\Oc_{C\times C} (a, a, 1))$ are sent to $\iota^*\vp = -\vp$ and a more interesting, even sequence invariant under the involution. Note, with $\vp$ as in \eqref{rf9hf983h0390fj3} the sign convention adopted here following Donagi-Witten \cite{DW2} is:
\begin{align}
dx^{\otimes a/2}dy^{\otimes b/2} = (-1)^{ab} dy^{\otimes b/2} dx^{\otimes a/2}
\label{rfh748f4hf983h0f3}
\end{align}
See e.g., that $dx\otimes dy = dy\otimes dx$ so $dx\wedge dy = -dy\wedge dx$. However, $dx^{1/2}\wedge dy^{1/2} = dy^{1/2}\wedge dx^{1/2}$ in keeping with the fermionic nature of $dx^{1/2}$ and $dy^{1/2}$ as implied by sign convention in \eqref{rfh748f4hf983h0f3}.

\begin{LEM}\label{rfh89hf8hf03j093}
Under the involution $\iota$ the sequence in \eqref{rfh793hf983hf083}, in the case $a = b$ is odd, splits into the following, short exact sequences:
\begin{align*}
&0 \lra \mathrm{Sym}^2H^0(K_C^{\otimes a}) \lra H^0\big(\Oc_{C\times C}(a,a, 1)\big)^- \lra 0 \lra 0;
\\
&0 \lra \wedge^2H^0(K_C^{\otimes a}) \lra H^0\big(\Oc_{C\times C}(a,a, 1)\big)^+ \lra H^0(K_C^{\otimes(a-1)})\lra 0.
\end{align*}
\end{LEM} 

\begin{proof}
By the K\"unneth formula and the decomposition of tensor-squares,
\begin{align}
H^0(\Oc_{C\times C}(a, a, 0))\cong \otimes^2H^0(K_C^{\otimes a}) \cong \big(\mathrm{Sym}^2H^0(K_C^{\otimes a})\big)\oplus \big(\wedge^2H^0(K_C^{\otimes a})\big).
\label{rfh894hf894fh0j}
\end{align}
Due to the pole order being simple (c.f., \eqref{rf9hf983h0390fj3}) the involution endomorphism $\iota^*\in \mathrm{End}~H^0(\Oc_{C\times C}(a, a, 1))$ induces a projection of the kernel $H^0(\Oc_{C\times C}(a, a, 0))$ onto the symmetric resp. antisymmetric square through \eqref{rfh894hf894fh0j} depending on whether the image of sections in $H^0(\Oc_{C\times C}(a, a, 0))$ are odd resp. even. The sheaf $\Delta_*K_C^{\otimes(a-1)}$ on $C\times C$ is supported on the diagonal. Hence there is only one section in $H^0(K_C^{\otimes(a - 1)})$ which is odd under the involution, the zero section. As such $\iota^*$ acts as the identity on $H^0(K_C^{\otimes(a-1)})$. Writing $H^0(\Oc_{C\times C}(a, a, 1)) = H^0(\Oc_{C\times C}(a, a, 0))^+ \oplus H^0(\Oc_{C\times C}(a, a, 0))^-$ the lemma now follows.
\end{proof}

\subsubsection{The pairing}\label{rfg78fg783f9h389}
We specialize here to the case $a = b = 3$ and $c = 1$ in \eqref{rfh793hf983hf083}. One of the key, technical innovations made by Donagi and Witten in \cite{DW2} lie in their construction of a duality pairing between SRS deformation parameters and the central extension $H^0(C\times C,\Oc_{C\times C}(3, 3, 1))$ in $\eqref{rfh793hf983hf083}|_{(a, b, c) = (3, 3, 1)}$. We describe this pairing presently. Over the base $\Abb^{0|q}_k$ for $q > 1$, deformation parameters for an SRS family $(\Scl\subset \Xc\ra\Abb^{0|q}_k)$ can, modulo $(\xi^3)$, be represented by triplets of forms $(\chi^m, \chi^n, \chi^{mn})_{m<n}$ where $\chi^m, \chi^n$ are the odd, {versal} deformation parameters representing classes in $H^1(C, K_C^{-1/2})$ and $\chi^{mn}$ are even parameters representing $H^1(C, K_C^{-1})$-torsors (c.f., Remark \ref{rfh983hf893fj390}). Now in the case $(a, b, c) = (3, 3, 1)$ the sequence in \eqref{rfh793hf983hf083} becomes
\[
0
\lra
H^0(\Oc_{C\times C}(3, 3, 0))
\lra
H^0(\Oc_{C\times C}(3, 3, 1))
\stackrel{res}{\lra}
H^0(K_C^{\otimes 2})
\lra 
0.
\]
Given any SRS deformation $(\chi^m, \chi^n, \chi^{mn})$, note $\chi^m\boxtimes \chi^n\in H^2(\Oc_{C\times C}(-1, -1, 0))$ pairs with elements in $\ker res$. As for $\chi^{mn}$, it can be written as a $K_C^{-1}$-valued $(0, 1)$-form in a given coordinate system and so can be paired with elements in $\img~res$. Through the theory of principal value distributions, these pairings extend to a pairing on $H^0(\Oc_{C\times C}(3, 3, 1))$ defined by,
\begin{align}
\left\langle \vp, \big(\chi^m, \chi^n, \chi^{mn}\big)\right\rangle 
\stackrel{\Delta}{=}
\int_{C\times C} \vp(\chi^m\boxtimes \chi^n)
-
2\pi \sqrt{-1}
\int_C
(res ~\vp)\chi^{mn}.
\label{rciuehcuiecioe}
\end{align}
The above pairing does not give rise to a pairing on cohomology since $\chi^{mn}$ does not define a cohomology class. However, in considering the triplet $(\chi^m, \chi^n, \chi^{mn})$ as gauge parameters, Donagi-Witten found that under a gauge transformation\footnote{c.f., \eqref{rfh89hf983fh093j0}.} by $(u^m, u^n, u^{mn})$ that
\begin{align}
\left(
\begin{array}{l}
\chi^m
\\
\chi^n
\\
\chi^{mn}
\end{array}
\right)
\longmapsto 
\left(
\begin{array}{l}
\chi^{\p m}
\\
\chi^{\p n}
\\
\chi^{\p mn}
\end{array}
\right)
=
\left(
\begin{array}{l}
\chi^m + \overline \pt_{Ber} u^m
\\
\chi^n + \overline \pt_{Ber} u^n
\\
\chi^{mn} + \overline \pt_{Ber} u^{nm} + \chi^m u^n - \chi^n u^m
\end{array}
\right)
\label{rg73gf983hf0j30}
\end{align}
where $\overline \pt_{Ber}$ is as in \eqref{rfh983hf983f09j93j}. A proof of the following can be found in \cite[\S3]{DW2}.

\begin{PROP}\label{rf873f893hf3j0}
The pairing in \eqref{rciuehcuiecioe} is invariant under gauge transformations in \eqref{rg73gf983hf0j30}, i.e., for any $\vp\in H^0(\Oc_{C\times C}(3, 3, 1))$ that 
\[
\left\langle \vp, \big(\chi^m, \chi^n, \chi^{mn}\big)\right\rangle = \left\langle \vp, \big(\chi^{\p m}, \chi^{\p n}, \chi^{\p mn}\big)\right\rangle
\]
if and only if the triplets $\big(\chi^m, \chi^n, \chi^{mn}\big)$ and $\big(\chi^{\p m}, \chi^{\p n}, \chi^{\p mn}\big)$ are related as in \eqref{rg73gf983hf0j30} for some $(u^m, u^n, u^{mn})$.
\qed
\end{PROP}

\begin{REM}
\emph{More precisely, Donagi and Witten in \cite{DW2} only consider the case $\Abb^{0|q}_k|_{q = 2}$ in Proposition \ref{rf873f893hf3j0}. The general case $\Abb^{0|q}_k$ for $q > 2$ was treated in \cite{BETTANAL}.}
\end{REM}
%\noindent
%As a consequence of Proposition \ref{rf873f893hf3j0} the construction in \eqref{rciuehcuiecioe} gives a pairing between spaces $H^0(\Oc_{|\Scl|\times |\Scl|}(3, 3, 1))\otimes Def_{\Abb^{0|2}_\Cbb}\Scl \ra \Cbb$. 

\subsection{Multiplication by the Szeg\"o kernel differential}\label{rhf784f89hf09j39f03}
Recall that our objective in this section is to obtain a relation between D'Hoker-Phong's period matrix formula in Proposition \ref{rfh984hf8hf0903fj93} and Donagi-Witten's gauge pairing in \eqref{rciuehcuiecioe}. As the quadratic component involves the Szeg\"o kernel differential, we firstly need to see how this kernel relates to elements in the central extension $H^0(\Oc_{C\times C}(3, 3, 1))$. Fix a spin curve $(C, K_C^{1/2})$. Moving to the product, the space of \emph{Szeg\"o kernel differentials} on $(C, K_C^{1/2})$ is $H^0(\Oc_{C\times C}(1, 1, 1))$. 

\begin{LEM}\label{rf8938f38jf03jf0}
In the case where $(C, K_C^{1/2})$ is generic and even, restriction to the diagonal gives an isomorphism $res: H^0(\Oc_{C\times C}(1, 1, 1))\stackrel{\sim}{\ra} \Cbb$.
\end{LEM}

\begin{proof}
Note, we cannot apply Lemma \ref{fh98hf89hf030f3} directly. It will still hold however and we can see this as follows. In general $H^0(\Oc_{C\times C}(1, 1, 1))$ will sit in a long exact sequence as in \eqref{rf748fh4f0j93}. By the K\"unneth decomposition we have,
\begin{align*}
H^0(\Oc_{C\times C}(1, 1, 0)) \cong \otimes^2 H^0(K_C^{1/2})
&&
\mbox{and}
&&
H^1(\Oc_{C\times C}(1, 1, 0)) \cong H^0(K_C^{1/2})\otimes H^1(K_C^{1/2}).
\end{align*}
Hence $h^i(\Oc_{C\times C}(1, 1, 0)) = 0$ for $i = 0, 1$ by virtue of $(C, K_C^{1/2})$ being generic and even. This lemma now follows from exactness in \eqref{rf748fh4f0j93}.
\end{proof}

\noindent 
On a generic, even spin curve $(C, K_C^{1/2})$ the \emph{normalised} Szeg\"o kernel differential is the image of the unit in $\Cbb$ under the restriction isomorphism $res$ in Lemma \ref{rf8938f38jf03jf0}. And so, to any generic, even $(C, K_C^{1/2})$ set $\mathbb S_{(C, K_C^{1/2})} \stackrel{\Delta}{=} res^{-1}1$. From Lemma \ref{rfh89hf8hf03j093} then we have:

\begin{COR}\label{rfg87gf87h3f98h397}
$\mathbb S_{(C, K_C^{1/2})}$ is even (i.e., invariant) under the involution on $C\times C$.\qed
\end{COR}

\noindent
Observe then, for any element $\psi\in H^0(\Oc_{C\times C}(a, b, 0))$, that we obtain a $\Cbb$-linear mapping $H^0(\Oc_{C\times C}(a, b, 0)) \ra H^0(\Oc_{C\times C}(a+1, b+1, 1))\stackrel{res}{\ra} H^0(K_C^{(a+b)/2})$ given by,
\begin{align}
\psi
\stackrel{=}{\longmapsto}
\psi\otimes  1
\stackrel{res^{-1}}{\longmapsto}
\psi\otimes \mathbb S_{(C, K_C^{1/2})}
\stackrel{cup~prod.}{\longmapsto}
\psi\cdot \mathbb S_{(C, K_C^{1/2})}
\stackrel{res}{\longmapsto}
res\left(\psi\cdot \mathbb S_{(C, K_C^{1/2})}\right).
\label{rhf97hf893hf83h0}
\end{align}
Alternatively, from the K\"unneth decomposition followed by the cup product we also have a map $H^0(\Oc_{C\times C}(a, b, 0)) \stackrel{\sim}{\ra} H^0(K_C^{a/2})\otimes H^0(K_C^{b/2}) \ra H^0(K_C^{(a+b)/2})$. This map will coincide with \eqref{rhf97hf893hf83h0} as both are defined through the cup product. Hence we have:

\begin{PROP}\label{rfh7f983hf98330}
On a generic, even spin curve $(C, K_C^{1/2})$ the following diagram will commute,
\[
\xymatrix{
\ar[d]_{\cdot\mathbb S_{(C, K_C^{1/2})}} H^0(\Oc_{C\times C}(a, b, 0)) \ar[rr]^{\cong} & & H^0(K_C^{a/2})\otimes H^0(K_C^{b/2})\ar[d]^{cup~prod.}
\\
H^0(\Oc_{C\times C}(a+1, b+1, 1))\ar[rr]^{res}
& &
H^0(K_C^{(a+b)/2})
}
\]
\qed
\end{PROP}

\begin{REM}
\emph{This description of what we have called the `normalized' Szeg\"o kernel can also be found in \cite{CODVIV}. It is instructive to compare Proposition \ref{rfh7f983hf98330} above with \cite[Thm. 6.6 and Rmk. 6.7]{CODVIV} where multiplication by the normalized Szeg\"o kernel is related to surjectivity of the period morphism.}
\end{REM}

\noindent
In the interests of this article we will further Proposition \ref{rfh7f983hf98330} by considering the case $(a, b, c) = (a, a, 1)$ and the involution invariant subspace $H^0(\Oc_{C\times C}(a, a, 1))^+\subset H^0(\Oc_{C\times C}(a, a, 1))$ characterized in Lemma \ref{rfh89hf8hf03j093}.

\begin{PROP}\label{rgf874gf794h89f}
Multiplication by $\mathbb S_{(C, K_C^{1/2})}$ gives a commutative diagram
\[
\xymatrix{
\ar[dd]H^0(\Oc_{C\times C}(a, a, -1))^+ \ar[r]&  H^0(\Oc_{C\times C}(a, a, 0))\ar[d]^{\mathbb S_{(C, K_C^{1/2})}}
\\
& H^0(\Oc_{C\times C}(a+1, a+1, 1))\ar[d]^{proj.}
\\
\wedge^2H^0(K_C^{3/2})\ar[r] 
&
H^0(\Oc_{C\times C}(a+1, a+1, 1))^+
}
\]
\end{PROP}

\begin{proof}
From Corollary \ref{rfg87gf87h3f98h397} and our sign convention in \eqref{rfh748f4hf983h0f3} the following composition vanishes
\begin{align*}
\xymatrix{
H^0(\Oc_{C\times C}(a, a, -1))^-\ar[r] & H^0(\Oc_{C\times C}(a, a, 0))\ar[d]^{\cdot\mathbb S_{(C, K_C^{1/2})}}
\\
& H^0(\Oc_{C\times C}(a+1, a+1, 1))\ar[d]^{proj.} 
\\
& H^0(\Oc_{C\times C}(a+1, a+1, 1))^+\ar[rrr]^{res}  & & & H^0(K^a_C)
}
%H^0(\Oc_{C\times C}(a, a, -1))
%&\lra 
%H^0(\Oc_{C\times C}(a, a, 0))
%\\
%&\stackrel{\cdot\mathbb S_{(C,K_C^{1/2})}}{\lra}
%H^0(\Oc_{C\times C}(a+1, a+1, 1))
%\\
%&\lra
%H^0(\Oc_{C\times C}(a+1, a+1, 1))^+
%\\
%&\stackrel{res}{\lra}
%H^0(K_C^a)
\end{align*}
whence the proposition.
\end{proof}

\begin{REM}\label{rf784gf78f938h9f}
\emph{Note that $H^0(\Oc_{C\times C}(a, a, -1))^+ = H^0(\Oc_{C\times C}(a, a, -2))^+$ in Proposition \ref{rgf874gf794h89f}.}
\end{REM}

%\begin{LEM}\label{rf9fh98h3f0j309j3}
%On a generic, even spin curve $(C, K_C^{1/2})$, multiplication by the normalised Szeg\"o kernel in Proposition \ref{rfh7f983hf98330} specialised to $(a, b) = (a, a)$ commutes with the projection onto the invariant subspace $H^0(\Oc_{C\times C}(a, a, 1))^+$ for any $a$.
%\end{LEM}
%
%
%\begin{proof}
%We want to argue that the multiplication $\cdot \mathbb S_{(C, K_C^{1/2})} : H^0(\Oc_{C\times C}(a, a, 0))\ra H^0(\Oc_{C\times C}(a+1, a+1, 1))$ will induce a map $\wedge^2H^0(K_C^{a/2}) \ra H^0(\Oc_{C\times C}(a+1, a+1, 1))^+$. To argue this, note that the Szeg\"o kernel differential is odd under the involution $\iota : C\times C\ra C\times C$. Similarly, via $H^0(\Oc_{C\times C}(a, a, 0))\cong \otimes^2H^0(K_C^{a/2}) \cong \wedge^2H^0(K_C^{a/2})\oplus Sym^2H^0(K_C^{a/2})$, the component in $\wedge^2H^0(K_C^{a/2})$ will also be odd under the involution. Hence the product $(\mu\wedge \nu)\cdot \mathbb S_{(C, K_C^{1/2})}$ will be even under the involution and so will lie in $H^0(\Oc_{C\times C}(a, a, 1))^+$ by definition.
%\end{proof}

\subsection{D'Hoker-Phong and Donagi-Witten}
Recall that the gauge pairing in \eqref{rciuehcuiecioe} is defined on sections in $H^0(\Oc_{C\times C}(3, 3, 1))$. As our ultimate applications will be in genus $g = 2$, we consider here the sheaf  $\Oc_{C\times C}(3, 3, 1)$ in genus $g = 2$. 

\begin{PROP}\label{rhf8hf98hf030}
Let $(C, K_C^{1/2})$ be a generic, even spin curve of genus $g = 2$. For any global section $\vp\in H^0(\Oc_{C\times C}(3, 3, 1))$ there will exist $\phi\in H^0(\Oc_{C\times C}(2, 2, 0))$ such that $res~\vp = res(\phi\cdot\mathbb S_{(C, K_C^{1/2})})$. 
\end{PROP}

\begin{proof}
In genus $g = 2$, a famous theorem by Max Noether states that the cup product $\otimes^2H^0(K_C)\ra H^0(K_C^2)$ will be surjective \cite[p. 233]{ARABGRIFF}. Hence, for any $\om\in H^0(K_C^2)$ there will exist some $\widetilde\om \in \otimes^2H^0(K_C)$ such that $\widetilde \om\mapsto \om$. This proposition now follows from commutativity in Proposition \ref{rfh7f983hf98330} specialized to $(a, b, c) = (2, 2, 0)$.
\end{proof}

\begin{REM}\label{rfh74fh9fh39jf093}
\emph{Since Max Noether's theorem also holds for genus $g = 1$ and non-hyperelliptic curves of genus $g> 2$, then so does Proposition \ref{rhf8hf98hf030} for such curves.}
\end{REM}

%\noindent 
%In the case elements in $H^0(\Oc_{C\times C}(3, 3, 1))^+$ we have:
%
%
%\begin{COR}\label{r4hf784hf9jf9j30}
%With the assumptions from Proposition \ref{rhf8hf98hf030}, let $(\om_\ell)_\ell$ be a basis of global differentials on $C$. Then for any $\vp\in H^0(\Oc_{C\times C}(3, 3, 1))^+$ there will exist $\ell_1, \ell_2$ such that $res~\vp = res\big(\om_{\ell_1}\cdot\mathbb S_{(C, K_C^{1/2})}\cdot \om_{\ell_2}\big)$.
%\end{COR}
%
%\begin{proof}
%From Lemma \ref{rf9fh98h3f0j309j3} and Proposition \ref{rhf8hf98hf030} we can find global differentials $\mu, \nu\in H^0(K_C)$ such that $res~\vp = res\big((\mu\wedge\nu)\cdot \mathbb S_{(C, K_C^{1/2})}\big)$. If $res~\vp\neq 0$ then $\mu\wedge \nu\neq 0$. As such $\{\mu, \nu\}$ will span a two-dimensional subspace of $H^0(K_C)$. Hence there will exist a basis of global differentials $(\om_\ell)_\ell$ and $\ell_1, \ell_2$ such that $\mu = \om_{\ell_1}$ and $\nu = \om_{\ell_2}$.
%\end{proof}

\subsubsection{The SRS period matrix revisited}
With the gauge pairing from \eqref{rciuehcuiecioe} recall from Proposition \ref{rf873f893hf3j0}: while each summand is not invariant under the gauge transformations in \eqref{rg73gf983hf0j30}, their signed difference will be gauge invariant. Now on a generic, even spin curve $(C, K_C^{1/2})$ and with global differentials $\om_i, \om_\ell\in H^0(K_C)$, let $\phi = \om_i\boxtimes \om_\ell \in H^0(\Oc_{C\times C}(2, 2, 0))$. Its image under the Szeg\"o kernel multiplication $\cdot \mathbb S_{(C, K_C^{1/2})}: H^0(\Oc_{C\times C}(2, 2, 0)) \ra H^1(\Oc_{C\times C}(3, 3, 1))$ from Proposition \ref{rfh7f983hf98330} gives an element which, by D'Hoker and Phong's formula in Proposition \ref{rfh984hf8hf0903fj93}, can be paired with  Kodaira-Spencer representatives of an SRS family. This gives the SRS period matrix. Clearly, it is \emph{not} invariant under the gauge transformations in \eqref{rg73gf983hf0j30}. From Proposition \ref{rf873f893hf3j0} and Proposition \ref{rfh7f983hf98330} however we have the following, gauge invariant extension of D'Hoker-Phong's formula:

\begin{PROP}\label{rfhhf98fj90j3f3f54f35f3}
Let $(\Scl \subset \Xc \ra \Abb^{0|n}_k)$ be a family of generic, even super Riemann surfaces with $\Scl$ modeled on $(C, K_C^{1/2})$. Let $(\om_\ell)_\ell$ be a basis of global differentials on the underlying curve $C\subset \Scl$; $\mathbb S_{(C, K_C^{1/2})}$ the normalized, Szeg\"o kernel differential and let $(\chi^m, \chi^n , \chi^{mn})$ be deformation parameters representing this family.\footnote{c.f., Theorem \ref{rfhhf893f0j39f03}.} Then the  quadratic component of $\widehat\Om_{\Xc/\Abb^{0|n}_k, i\ell}$ is given by its $(i,\ell)^{(m, n)}$-th entry,
\begin{align}
\widehat\Om^{mn}_{\Xc/\Abb^{0|n}_k, i\ell}&
\notag
\\
=~&
\int_{C\times C}
\big(\om_i\cdot \mathbb S_{(C, K_C^{1/2})}\cdot\om_\ell\big)\big(\chi^m\boxtimes \chi^n\big)
-
2\pi\sqrt{-1}
\int_C
\big(\om_i\cdot\om_\ell\big)\chi^{mn}
\notag
\\
=~&
\left\langle 
\om_i\cdot \mathbb S_{(C, K_C^{1/2})}\cdot \om_\ell, 
\big(\chi^m, \chi^n, \chi^{mn}\big)
\right\rangle
\label{rhf784gf784fh9h3f083}
\end{align}
where `$\cdot$' denotes the cup product on cohomology.\qed
\end{PROP}

\begin{REM}
\emph{It is instructive to compare Proposition \ref{rfhhf98fj90j3f3f54f35f3} above with remarks by Witten in \cite[p. 103]{WITTRS}, suggesting a gauge-invariant formula for the SRS period matrix. We claim this is precisely Proposition \ref{rfhhf98fj90j3f3f54f35f3}.}
\end{REM}

\section{Deformations of super Riemann surfaces}
\label{rfg674gf84f9h843f}

\subsection{Schiffer variations}
It is a classical result in the geometry of compact, complex manifolds that the versal deformation space of its complex structures can be generated by Schiffer variations at sufficiently many points. We will show how these ideas generalise to the deformation theory of super Riemann surfaces. We begin with the case of curves, following \cite[Ch. XI]{ARABGRIFF}. 

\subsubsection{On curves}
Let $C$ be a curve, $T_C$ its tangent sheaf and $P\in C$ a point. Viewing $P$ as a closed subspace of $C$, its ideal sheaf sequence tensored with $T_C(P)$ gives the following sequence\footnote{see e.g., \cite[p. 296]{HARTALG} for a more general case} of sheaves on $C$,
\begin{align}
0 \lra T_C \lra T_C(P) \lra \underline\Cbb_P\lra 0
\label{fvgf97h389fh830fh}
\end{align}
where $\underline \Cbb_P$ is the skyscraper sheaf at $P$. 
This sequence gives rise to a long exact sequence on cohomology containing the following piece,
\begin{align}
0 
\lra 
H^0(T_C) 
\lra  
H^0(T_C(P))
\lra
H^0(\underline \Cbb_P)
\stackrel{\dt}{\lra}
H^1(T_C)
%\xymatrix{
%0 \ar[r] & H^0(T_C) \ar[r] & H^0(T_C(P)) \ar[r] & H^0(T_C(P)\otimes_{\Oc_C}\underline \Cbb_P) \ar[r] & H^1(T_C).
%}
\label{rf83f03jf09j3}
\end{align}
In the case where the genus of $C$, $g$, satisfies $g > 1$ we know that $h^0(T_C) = h^0(T_C(P)) = 0$ since both  $T_C$ and $T_C(P)$ have negative degree.\footnote{\label{rhf4hf98h89fhf0}From e.g., \cite[p. 296]{HARTALG}, $\deg T_C(P) = \deg T_C + 1$ and $\deg T_C = 2 - 2g$. And so $\deg T_C$ and $\deg T_C(P)$ are both negative for $g > 1$. }
Therefore the boundary map on cohomology gives an embedding $H^0(\underline \Cbb_P)\subset H^1(T_C)$. Now generally, for any skyscraper sheaf $\Fc_P$ one has $H^0(\Fc_P)\cong i_P^*\Fc_P/\mathfrak m_P$, the stalk of $\Fc_P$ at $P$ modulo its maximal ideal.\footnote{Here $i_P : \{P\}\subset C$ is the embedding of the point $P$.} In the case $\Fc_P = \underline \Cbb_P$ we have therefore $H^0(\underline\Cbb_P) \cong \Cbb$. The image of the generator $1\in \Cbb$ under the boundary map $\dt$ in \eqref{rf83f03jf09j3} is referred to as the \emph{Schiffer variation of $C$ at $P$}. Denoting this by $(C\subset \Xc_P \ra \Cbb)$, the Kodaira-Spencer class can be described as follows:\footnote{c.f., \cite[p. 175]{ARABGRIFF}} let $\{U, V\}$ be a covering of $C$ where $U$ is a sufficiently small, open neighbourhood of $P$ and $V$ is the complement. Let $z$ be the local coordinate on $U$ centered at $P$. Then with respect to the cover $\{U, V\}$,
\begin{align}
(\kappa_{\Xc_P, 0})_{UV} = \frac{1}{z} \frac{\pt}{\pt z}
\label{rfh93fh983hf93h9}
\end{align}
We now have the following useful result.

\begin{LEM}\label{rfhu3hf93hf8h30}
As a $(0, 1)$-form, the Kodaira-Spencer class of a Schiffer variation at $P$ can be represented by the delta distribution $\pi\dt(z)d\overline z\frac{\pt}{\pt z}$ where $z$ is centered at $P$.\footnote{\label{rf784f9hf893h3434}The $\dt$ appearing in this statement has no relation to the boundary map $\dt$ in \eqref{rf83f03jf09j3}. Hopefully no confusion will arise over this blatant abuse of notation.}
\end{LEM}

\begin{proof}
From the theory of principal value distributions we have the classical relation deftly employed by Donagi-Witten in \cite{DW2},
\begin{align}
\frac{\pt}{\pt \overline z}\frac{1}{z} = \pi\dt(z).
\label{rfh893hf89h38f30j9f03}
\end{align}
Now generally for a sheaf $\Fc$ on a complex space $X$ one has the Dolbeault isomorphism $H^i(X, \Om^j_X\otimes \Fc)\cong H^{j, i}_{\overline \pt}(\Fc)$. To illustrate in the case $i = 1, j = 0$, an element $f\in H^1(X, \Fc)$ can be represented by the $1$-cocycle $f_{ij} = g_j - g_i$ with respect to an open cover $(U_i)_i$. Since $f_{ij}$ is holomorphic, it will be $\overline \pt$-closed giving $\overline \pt g_i = \overline \pt g_j$. Hence there will exist a some global, $\overline\pt$-closed $h\in \Gam^{0, 1}(X, \Fc)$ such that $h|_{U_i} = \overline\pt g_i$. 
Specializing now to the case of interest in the statement of this lemma, let $(C\subset \Xc_P\ra \Cbb)$ denote the Schiffer variation at $P$. Recall the representation of its Kodaira-Spencer class $\kappa_{\Xc_P, 0}$ in \eqref{rfh93fh983hf93h9} with respect to the cover $\{U, V\}$. Under the Dolbeault isomorphism, it will be represented by $\overline\pt\left(\frac{1}{z}\right)\frac{\pt}{\pt z}$. The lemma follows from \eqref{rfh893hf89h38f30j9f03}.
\end{proof}

\noindent
For $h^1(T_C)$-many points on $C$ in general position $(P_h)_{h= 1, \ldots, h^1(T_C)}$ see that on cohomology $H^0(\oplus_h\underline \Cbb_{P_h})\cong \oplus_h H^0(\underline \Cbb_{P_h})\cong \Cbb^{h^1( T_C)}$. Since the boundary map $\dt$ in \eqref{rf83f03jf09j3} is a monomorphism whenever the genus $g>1$ it follows that $\dt$ gives an isomorphism $H^0(\oplus_h\underline \Cbb_{P_h})\cong H^1(T_C)$. Hence, Schiffer variations generate the versal deformation space of complex structures on $C$.

\subsubsection{On super Riemann surfaces}\label{rfh73fg9h89f3fh903j}
In the case of super Riemann surfaces $\Scl$, recall from \S\ref{fnckbcubcuiencioe} that there are two types of deformations or families here: even types and odd types. In this article we are only interested in the odd types, being families $(\Scl\subset \Xc\ra \Abb^{0|q}_k)$. For $\Scl$ modeled on a given spin curve $(C, K_C^{1/2})$ the odd, versal deformation space of $\Scl$ is the cohomology $H^1(C, K_C^{-1/2})$. With this given, note that we have an isomorphism $K_C^{-1/2} \cong T_C\otimes K_C^{1/2}$. At a point $P\in C$ then, tensoring the sequence \eqref{fvgf97h389fh830fh} with our spin structure gives
\begin{align}
0 \lra K_C^{-1/2} \lra K_C^{-1/2}(P) \lra \underline \Cbb_P\lra 0.
\label{hf849hf984hf903}
\end{align}
On cohomology the image of the boundary map $\dt : H^0(\underline \Cbb_P) \ra H^1(K_C^{-1/2})$ defines the  class of deformations termed \emph{odd Schiffer variations}. Arguing as in Lemma \ref{rfhu3hf93hf8h30} and using the representations in \eqref{rfh89hf983fh093j0}, \eqref{rfiuhf983hf0j390} we have:

\begin{LEM}\label{rfg748fg784hf89f30}
Fix a super Riemann surface $\Scl$ modeled on $(C, K_C^{1/2})$ with genus $g>1$. Let $(\Scl \subset \Xc_P \ra \Abb_k^{0|1})$ denote an odd Schiffer variation of $\Scl$ at a point $P\in C$. Then its Kodaira-Spencer class can be represented by the smooth, $(0,1)$-Berezinian distribution $\pi\dt(z)\q~[dzd\overline z|d\q]$ where $z$ is a local coordinate centered at $P$.\footnote{c.f., f.t. \ref{rf784f9hf893h3434}.}\qed
\end{LEM}

\noindent
As for whether the odd deformation space can be generated by odd Schiffer variations, note firstly that $\deg K_C^{-1/2} = g - 1 < 0$ for genus $g> 1$. It will therefore not admit any global sections. For $K_C^{-1/2}(P)$ however, this is only true in genus $g > 2$.\footnote{See f.t. \ref{rhf4hf98h89fhf0}.} In genus $g = 2$, $\deg K_C^{-1/2} = 0$. Nevertheless, note that:
%The deformation space $H^1(K_C^{-1/2})$ can nevertheless be generated by the odd Schiffer variations however since: 
(i) $h^1(K_C^{-1/2}) = 2$ in genus $g= 2$ and (ii) $\deg K_C^{-1/2}(P+Q) < 0$ for distinct points $P, Q\in C$. From the cohomology sequence induced from \eqref{hf849hf984hf903} we arrive therefore at the analogue to the classical case:

%Arguing as above, note that for genus $g(C) > 1$, both $K_C^{-1/2}$ and $K_C^{-1/2}(P)$ are line bundles on $C$ of negative degree. As such they have no global sections, leading to an embedding of cohomology spaces $H^0(K_C^{-1/2}(P)\otimes_{\Oc_C}\underline \Cbb_P) \subset H^1(K_C^{-1/2})$. Now since $K_C^{1/2}$ is a spin structure we have $K_C\cong K_C^{1/2}\otimes K_C^{1/2}$. Hence that $K_C^{-1}\cong  K_C^{-1/2}\otimes K_C^{-1/2}$. This gives an identification on local sections $\pt/\pt z = \pt/\pt \q\otimes \pt/\pt\q$, where $(z|\q)$ are local coordinates on the super Riemann surface $\Scl$. Consequently we find here, similarly to earlier, that $H^0(K_C^{-1/2}(P)\otimes_{\Oc_C}\underline \Cbb_P) \cong \Cbb$ is generated by the global section which, in the coordinates $(z|\q)$, is given by $\frac{1}{z}\frac{\pt}{\pt \q}$. Its image under the boundary map $\dt : H^0(K_C^{-1/2}(P)\otimes_{\Oc_C}\underline \Cbb_P)\ra H^1(K_C^{-1/2})$ is the \emph{odd} Schiffer variation of a super Riemann surface $\Scl$ at the point $P\in C$. Arguing as in Lemma \ref{rfhu3hf93hf8h30} will lead to:
%
%
%
%
%\noindent
%Regarding the odd, versal deformation space we have:

\begin{LEM}\label{rfh89hf983hf0309}
Let $\Scl$ be a super Riemann surface modeled on $(C, K_C^{1/2})$ with genus $g> 1$. It's odd, versal deformation space is generated by odd Schiffer variations at $h^1(C, K_C^{1/2})$-many points on $C$ in general position. \qed
\end{LEM}

\subsection{Associated points and split gauge}\label{rfg78gf874gf7h93}
Let $\Fc$ be an invertible sheaf on a compact, Riemann surface. The total number zeros and poles of a generic section of $\Fc$, counted with multiplicity, adds up to the degree of $\Fc$. For a spin curve $(C, K_C^{1/2})$, the degree of $K_C^{1/2}$ is $g-1$, where $g$ is the genus of $C$. Let $\mathbb S$ be a Szeg\"o kernel differential on $(C, K_C^{1/2})$. To any point $P\in C$, see that $\mathbb S(P, y)$ will be a meromorphic section of $K_C^{1/2}$ with a simple pole at $y = P$.\footnote{\label{rfg674rgf6gf}In more invariant terms, recall from \S\ref{rhf784f89hf09j39f03} that $\mathbb S$ is an element of $H^0(C\times C, \Oc_{C\times C}(1, 1, 1))$. With $i_P : \{P\}\times C\subset C\times C$ the closed embedding of the point $P$, see that $i_P^*\Oc_{C\times C}(1, 1, 1) \cong K_C^{1/2}(P)$. Hence $i_P^*\mathbb S$ will be a meromorphic section of $K_C^{1/2}$. In a local coordinate $y$, $(i_P^*\mathbb S)(y) = \mathbb S(P, y)$.}
It follows that $\mathbb S(P, y)$ will have $g$-many zeroes on $C$, counted with multiplicity. Hence on the product $C\times C$ there will exist, for any $P\in C$ fixed, a collection of points $\{P_\nu\}_{\nu =1, \ldots, g}$ such that $\mathbb S(P, P_\nu) = 0$. Hawley and Schiffer in a classical article \cite{HAWLEYSCHIFF} give an explicit construction of Szeg\"o kernel differentials on spin curves and its locus of zeroes (see \cite[\S III]{HAWLEYSCHIFF}). Pairs $(P, Q)\in C\times C$ for which $\mathbb S(P, Q) = 0$ are referred to as \emph{associated points} with respect to $\mathbb S$. For our purposes we will only need to know the following.
%To a given point $P\in C$ we refer to the set of all $Q\in C$ such that $\mathbb S(P, Q) = 0$ as the \emph{associated point set of $P$ w.r.t. $\mathbb S$}.

\begin{LEM}\label{rfh78hf893hf93j0f93j}
Let $(C, K_C^{1/2})$ be a spin curve where $C$ has genus $g> 0$ and let $\mathbb S$ be a Szeg\"o kernel differential on $(C, K_C^{1/2})$. Then there will always exist at least two distinct points $P, Q\in C$ such that $\mathbb S(P, Q) = 0$.
\end{LEM}

\begin{proof}
Let $P\in C$ be fixed. With the genus $g> 0$, the degree of $K_C^{1/2}(P)$ will be non-zero and positive. Hence the associated point set $\{P_\nu\}_\nu$ will be non-empty. Since $y = P$ is a simple pole for the kernel $\mathbb S(P, y)$, it obviously cannot be a zero. And since $\mathbb S(P, P_{\nu_0}) = 0$ for any $P_{\nu_0}\in \{P_\nu\}_\nu$ by definition, $y = P_{\nu_0}$ is a zero for $\mathbb S(P, y)$. Therefore $P \neq P_{\nu_0}$ for any $P_{\nu_0} \in \{P_\nu\}_\nu$. 
\end{proof}

\section{Supermoduli splitting}
\label{rfgt6gf48fh584}
\subsection{The obstruction class}

\subsubsection{Generalities} 
A complex supermanifold $\Xfr$ is split if it is isomorphic to its split model. For $\Xfr$ modeled on $(X, T_{C, -}^*)$ its first or `primary' obstruction to splitting is a class $\om_\Xfr$ in the sheaf cohomology $H^1(X, \wedge^2T^*_{X, -}\otimes T_X)$. We refer the reader to \cite{BETTHIGHOBS, BETTOBSTHICK} and references therein for further details on supermanifold obstruction theory. With the identification $H^1(X, \wedge^2T^*_{X, -}\otimes T_X) \cong \mathrm{Ext}^1(\wedge^2T_{X, -}, T_X)$ see that $\om_\Xfr$ corresponds, up to equivalence, to a central extension of sheaves
\begin{align}
0 \lra T_X \lra \Ec_{\om_\Xfr} \lra \wedge^2T_{X, -}\lra 0.
\label{rhf893hf83f0j93}
\end{align}
Or dually, to:
\begin{align}
0 \lra \wedge^2T^*_{X, -} \lra \Ec^*_{\om_\Xfr} \lra T^*_X\lra 0.
\label{rhf893hf89f093}
\end{align}
In particular, $\om_\Xfr = 0 \iff$ the central extension \eqref{rhf893hf83f0j93}, equivalently \eqref{rhf893hf89f093}, splits. While the vanishing of $\om_\Xfr$ is generally insufficient to conclude the existence of a splitting, in the case $\dim_-\Xfr = 2$ it is both necessary and sufficient, as shown for instance in \cite[p. 191]{YMAN}. That is:

\begin{LEM}\label{rfu3f983hfj390f444}
Let $\Xfr$ be a $(p|2)$-dimensional, complex supermanifold. Then $\Xfr$ is split if and only if $\om_\Xfr = 0$.
\qed
\end{LEM}

\subsubsection{On supermoduli space}
We specialize now to the case where $\Xfr = \Mfr_g$ is the genus $g$, SRS moduli space. This is a superspace which is modeled on the genus $g$, spin moduli space $\Scl\Mcl_g$ and a stacky $\Zbb_2$-bundle $\mathbb E^*_g\ra \Scl\Mcl_g$. The fiber of $\mathbb E_g$ at a spin curve $(C, K_C^{1/2})$ is the odd, versal deformation space of the super Riemann surface $\Scl$ modeled on $(C, K_C^{1/2})$. Recalling this from \S\ref{fnckbcubcuiencioe} it is,
\begin{align}
\mathbb E_g\big|_{(C, K_C^{1/2})} = H^1(C, K_C^{-1/2}).
\label{fnckbvbeiuvbek}
\end{align}
Let $\Ec^*_g$ denote the sheaf of holomorphic sections of $\mathbb E^*_g$. The primary obstruction to splitting $\Mfr_g$ is a class in $H^1(\Scl\Mcl_g, \wedge^2\Ec_g^*\otimes T_{\Mcl_g})$.\footnote{\label{rgf784gf83hf893}Note, since $\Scl\Mcl_g \ra \Mcl_g$ is a finite covering of the moduli space of genus $g$ curves $\Mcl_g$, we have identified $T_{\Scl\Mcl_g}$ with $T_{\Mcl_g}$.} Hence from \eqref{rhf893hf89f093} it will correspond to a central extension $\Ec_{\om_{\Mfr_g}}$. We will work with the dual sequence,
\begin{align}
0
\lra
\wedge^2 \Ec_g^*
\lra 
\Ec_{\om_{\Mfr_g}}^*
\lra
T^*_{\Mcl_g}
\lra
0.
\label{rfh74fg98hf039j}
\end{align}
%While $\mathbb E^*_g$ need not define a vector bundle over $\mathcal S\Mcl_g$, its tensor square, and hence antisymmetric square will define a vector bundle. This is because the automorphism $\mathbb E^*_g\ra \mathbb E^*_g$ defined by the non-trivial generator in $\Zbb_2$ will act trivially on the tensor square and hence also on the antisymmetric square. As a result the sequence in \eqref{rfh74fg98hf039j} is of vector bundles on $\Scl\Mcl_g$.
Donagi and Witten in \cite[Proposition 3.1]{DW2} prove: 

\begin{PROP}\label{rfg78gf78hf893}
The restriction of \eqref{rfh74fg98hf039j} to a spin curve $(C, K_C^{1/2})$ is naturally isomorphic to the even sequence in Lemma \ref{rfh89hf8hf03j093}, i.e., that,
\[
\xymatrix{
0 \ar[r] & \wedge^2 \Ec_g^*\big|_{(C, K_C^{1/2})} \ar[rr]\ar[d]^\cong & & \Ec_{\om_{\Mfr_g}}^*\big|_{(C, K_C^{1/2})} \ar[d]^\cong  \ar[rr] & & T^*_{\Mcl_g}\big|_{(C, K_C^{1/2})} \ar[d]^\cong \ar[r]  & 0
\\
0 \ar[r]  & \wedge^2H^0(K_C^{3/2}) \ar[rr] & & H^0(\Oc_{C\times C}(3, 3, 1))^+ \ar[rr] & & H^0(K_C^2)\ar[r]  & 0
}
\]
\qed
\end{PROP}

\noindent 
In Lemma \ref{rfh89hf8hf03j093} the involution on $C\times C$ served to split the sequence in $\eqref{rfh793hf983hf083}|_{(a, b, c) = (a, a, 1)}$ into a trivial, odd part and a non-trivial, even part. The relevance of this sequence to sheaves on supermoduli space is illustrated in Proposition \ref{rfg78gf78hf893}. More generally, a direct analogy on supermoduli space is described presently following Donagi-Witten \cite[\S3]{DW2}. Let $\Scl\Mcl_{g, 1} \stackrel{p}{\ra} \Scl\Mcl_g$ be the universal curve forgetting a point and $K_p$ the relative cotangent sheaf of $p$. It is a sheaf on $\Scl\Mcl_{g, 1}$ with direct image at a spin curve $(C, K_C^{1/2})\in \Scl\Mcl_g$ given by $(p_*K_p)_{(C, K_C^{1/2})} = H^0(C, K_C)$. Since we are over the spin moduli space it is possible to form the half-integer powers $K_p^{a/2}$. For odd $a$, $K_p^{a/2}$ will be the sheaf of holomorphic sections of a stack-theoretic vector bundle over $\Scl\Mcl_{g, 1}$ with direct image $(p_*K_p^{a/2})_{(C, K_C^{1/2})} = H^0(C, K_C^{a/2})$.

\begin{REM}\label{rfh893hf39fj3f30}
\emph{Comparing with \eqref{fnckbvbeiuvbek}, note that $\Ec^*_g \cong p_*K_p^{3/2}$.}
\end{REM}

\noindent 
With the morphism $p$ we can form the Cartesian pullback to get the diagram
\begin{align}
\xymatrix{
\ar[d]_{p_1} \Scl\Mcl_{g, 1}\times_{\Scl\Mcl_g}\Scl\Mcl_{g, 1} \ar[rr]^{p_2} & & \Scl\Mcl_{g, 1}\ar[d]_p
\\
\Scl\Mcl_{g, 1}\ar[rr]^p & & \Scl\Mcl_g
}
\label{rfh3hf98hf8030}
\end{align}
With $\Ec^*_g$ the sheaf of holomorphic sections of the bundle $\mathbb E^*_g \ra \mathcal S\mathcal M_g$ characterized in \eqref{fnckbvbeiuvbek}, it can be pulled back to the product $\Scl\Mcl_{g, 1}\times_{\Scl\Mcl_g}\Scl\Mcl_{g, 1}$ through the projections $(p_i)_{i = 1, 2}$ in \eqref{rfh3hf98hf8030}. As in \eqref{rf7fh398fh3f03} define,
\begin{align}
\widetilde\Oc (a, b, c) \stackrel{\Delta}{=} p_1^*K_p^{a/2} \otimes p_2^* K_p^{b/2} (c\Delta)
\label{rfg8gf93h8390}
\end{align}
where $\Delta : \Scl\Mcl_{g, 1} \subset \Scl\Mcl_{g, 1}\times_{\Scl\Mcl_g}\Scl\Mcl_{g, 1}$ is the diagonal embedding. The globalization of $\eqref{rfh793hf983hf083}|_{(a, b, c) = (3, 3, 1)}$ to the spin moduli space is now an exact sequence of sheaves,
\begin{align}
0
\lra
\widetilde \Oc(3, 3, 0)
\lra 
\widetilde \Oc(3, 3, 1)
\lra 
\Delta_*K_p^2
\lra
0.
\label{rhf93hf93hf83h0}
\end{align}
On $\Scl\Mcl_{g, 1}\times \Scl\Mcl_{g, 1}$ the involution swapping pointed, spin curve moduli gives an involution on the subspace $\Scl\Mcl_{g, 1}\times_{\Scl\Mcl_g}\Scl\Mcl_{g, 1}$. Just as in Lemma \ref{rfh89hf8hf03j093}, the involution will split \eqref{rhf93hf93hf83h0} into a trivial, odd piece and the following, non-trivial, even piece
\begin{align}
0
\lra
\widetilde \Oc(3, 3, 0)^+
\lra 
\widetilde \Oc(3, 3, 1)^+
\lra 
\Delta_*K_p^2
\lra
0.
\label{rfj039jf093jf90jf3j}
\end{align}
Now from \eqref{rfh3hf98hf8030} note that $pp_1 = pp_2$ by definition of the cartesian pullback. Let $pp = pp_1 = pp_2$ denote the projection onto $\Scl\Mcl_g$. Note that\footnote{c.f., f.t., \ref{rgf784gf83hf893}.}  $pp_*\Delta_*K_p^2\cong T_{\Mcl_g}^*$ and, with Remark \ref{rfh893hf39fj3f30}, that $pp_*\widetilde\Oc(3, 3, 0)^+\cong \wedge^2\Ec^*_g$. These isomorphisms and the fiber-wise isomorphisms in Proposition \ref{rfg78gf78hf893} give rise to a morphism $pp_*\widetilde\Oc(3, 3, 1) \ra \Ec^*_{\om_{\Mfr_g}}$. By the short five lemma,  $pp_*\widetilde\Oc(3, 3, 1) \ra \Ec^*_{\om_{\Mfr_g}}$ will be an isomorphism. Hence we have arrived at the following result (see \cite[Theorem 3.2]{DW2}):

\begin{THM}\label{rhf983hf8030fj30}
The extension class of $pp_*\eqref{rfj039jf093jf90jf3j}$ coincides, up to sign, with the primary obstruction to splitting the unpointed, genus $g$ supermoduli space.\qed
\end{THM}

\begin{REM}
\emph{We have only considered the unpointed supermoduli space $\Mfr_g$. Theorem \ref{rhf983hf8030fj30} and the constructions leading up to it can be generalized however to $\Mfr_{g, n}$ for any number of markings $n$.}
\end{REM}

\subsection{Splitting in genus two}
We arrive now at the main result of this article.

\begin{THM}\label{rhf893h8f3h0fj39}
The generic, even component of the (unpointed) supermoduli space of curves is split as a superspace in genus $g = 2$. 
\end{THM}

\subsubsection{Fiber-wise splitting}\label{rfg78gf87hf9h389f9345}
We begin our proof of Theorem \ref{rhf893h8f3h0fj39} with the analogue at a fiber, i.e., at generic, even spin curve\footnote{c.f.,  Proposition \ref{rfg78gf78hf893}} $(C, K_C^{1/2})$. 

\begin{LEM}\label{rfg784gf784hf9h383}
Let $(C, K_C^{1/2})$ be a generic, even spin curve of genus $g = 2$. For any $\om \in H^0(K_C^2)$ let $\widetilde \om, \widetilde \om^\p\in H^0(\Oc_{C\times C}(2, 2, 0))$ be lifts of $\om$, whose existence is guaranteed by Proposition \ref{rhf8hf98hf030}. Then $(\widetilde \om - \widetilde \om^\p)\cdot \mathbb S_{(C, K_C^{1/2})} \in H^0(\Oc_{C\times C}(3, 3, -1))^+$.
\end{LEM}

\begin{proof}
By Corollary \ref{rfg87gf87h3f98h397} the normalized Szeg\"o kernel $\mathbb S_{(C,K_C^{1/2})}$ is even under the involution on $C\times C$. Hence $\eta\cdot \mathbb S_{(C, K_C^{1/2})}$ will be even iff $\eta$ is even. Now since $\widetilde\om, \widetilde\om^\p\in H^0(\Oc_{C\times C}(2, 2, 0))$ lift $\om$, their difference lies in the kernel $H^0(\Oc_{C\times C}(2, 2, -1)) = \ker\{H^0(\Oc_{C\times C}(2, 2, 0)) \ra H^0(K_C^2)\}$. By Remark \ref{rf784gf78f938h9f} now, $H^0(\Oc_{C\times C}(2, 2, -1))^+ = H^0(\Oc_{C\times C}(2, 2, -2))^+$. Therefore, under multiplication by $\mathbb S_{(C, K_C^{1/2})}$, $\widetilde\om - \widetilde \om^\p$ will map into $H^0(\Oc_{C\times C}(3, 3, -1))^+$.
\end{proof}

\noindent
We will now argue $h^0(\Oc_{C\times C}(3, 3, -1))^+ = 0$ in genus $g = 2$ using deformation theory. Recall firstly that, in genus $g = 2$, $h^1(K_C^{1/2}) = 2$. By Lemma \ref{rfh89hf983hf0309} the odd, versal deformation space of a super Riemann surface $\Scl$ modeled on $(C, K_C^{1/2})$ will be generated by Schiffer variations at two distinct points. By \eqref{rfh983hf983f09j93j} and Theorem \ref{rfhhf893f0j39f03} this means the following: if $P, Q\in C$ denote distinct points and $\chi^P, \chi^Q$ denote the respective (odd) Schiffer variations, then for \emph{any} odd, versal deformation parameters $\chi^m, \chi^n$ of $\Scl$, there exist Berezinian forms\footnote{In \cite{DW2, WITTRS} these would be odd, superconformal vector fields}  $u^m, u^n$ such that 
\begin{align}
\chi^m + \overline\pt_{Ber}u^m = \chi^P
&&
\mbox{and}
&&
\chi^n + \overline \pt_{Ber}u^n = \chi^Q.
\label{rg87gf783h89fh309f390}
\end{align}
Expressions for $\chi^P$ and $\chi^Q$ can be obtained from Lemma \ref{rfg748fg784hf89f30}. 
%It leads to the following observation. 
%%We have moreover the following observation which will be crucial in constructing the fiber-wise splitting.
%
%\begin{LEM}\label{rfg874gf784fh8hf893}
%The vector fields $u^m, u^n$ in \eqref{rg87gf783h89fh309f390} can be chosen such that 
%\[
%u^m \chi^Q = u^n\chi^P = 0.
%\]
%\end{LEM}
%
%\begin{proof}
%%Since $\chi^P$ resp. $\chi^Q$ are supported only at the points $P$ resp., $Q$, it suffices to argue $u^n$ resp. $u^m$ can be chosen such that they vanish at $P$ resp. $Q$. 
%Focussing on the equation in \eqref{rg87gf783h89fh309f390} for $\chi^Q$, let $\widetilde u$ be a solution to $\overline\pt_{Ber}\widetilde u = \chi^Q - \chi^n$. Set $u = \widetilde u - \widetilde u(P)$. Then by construction $\overline \pt_{Ber} u = \overline \pt_{Ber}\widetilde u$ and $u(P) = 0$. Since the support of $\chi^P$ is $\{P\}$ we have $u\chi^P = 0$ and we can take $u^n = u$. A similar argument allows for finding $u^m$ solving the equation in  \eqref{rg87gf783h89fh309f390} for $\chi^P$ and satisfying $u^m\chi^Q = 0$. 
%\end{proof}
%\noindent
By Lemma \ref{rfh78hf893hf93j0f93j} we can take $P$ and $Q$ to be the associated points of the normalized  Szeg\"o kernel differential $\mathbb S_{(C, K_C^{1/2})}$. 

\begin{LEM}\label{fbcvyvyubeiucievuev}
For any $\phi \in H^0(\Oc_{C\times C}(2, 2, 0))$ there exists $\zeta\in H^0(\Oc_{C\times C}(3, 3, 0))$ such that 
\[
\int_{C\times C}\phi\cdot \mathbb S_{(C,K_C^{1/2})}(\chi^P\boxtimes \chi^Q) = \int_{C\times C}\zeta(\chi^P\boxtimes \chi^Q)
\]
where $\int_{C\times C}\zeta(\chi^P\boxtimes \chi^Q)$ is defined through Serre duality.
\end{LEM}

\begin{proof}
This can readily be seen from the characterization of odd Schiffer variations in Lemma \ref{rfg748fg784hf89f30} along with our assumption $\mathbb S_{(C, K_C^{1/2})}(P, Q) = 0$.\footnote{A similar idea can be found in \cite[\S7.2, \S7.5]{HOKERPHONG1}.}
\end{proof}

\begin{PROP}\label{fjbvhfrbvrbvkjnk}
Let $\Xc \ra \Abb^{0|q}_k$ be a family of genus $g = 2$, generic, even super Riemann surfaces. For any $\vp\in H^0(\Oc_{C\times C}(3, 3, 1))$ and parameters $(\chi^m, \chi^n, \chi^{mn})$ representing the SRS family, there exists $\gam\in H^0(\Oc_{C\times C}(3, 3, 0))$ such that 
\[
\int_{C\times C} \vp\big(\chi^m\boxtimes \chi^n\big) = \int_{C\times C} \gam\big(\chi^m\boxtimes \chi^n\big)
\]
where the pairing $\int_{C\times C} \gam\big(\chi^m\boxtimes \chi^n\big)$ is defined through Serre duality.
\end{PROP}

\begin{proof}
Let $\vp\in H^0(\Oc_{C\times C}(3, 3, 1))$ so that $res~\vp\in H^0(K_C^2)$. By Proposition \ref{rhf8hf98hf030} there exists a lift $\widetilde{res}~\vp \in H^0(\Oc_{C\times C}(2, 2, 0))$ and so we can form the difference $\widetilde\zeta = \vp - (\widetilde{res}~\vp)\cdot \mathbb S_{(C, K_C^{1/2})}$. By construction $res~\gam = 0$ so that $\gam\in H^0(\Oc_{C\times C}(3, 3, 0))$. With $P, Q$ the associated points of $\mathbb S_{(C, K_C^{1/2})}$ and $\chi^P, \chi^Q$ the odd Schiffer variations respectively, we have from Lemma \ref{fbcvyvyubeiucievuev},
\begin{align}
\int_{C\times C}(\widetilde \zeta + \zeta)\big(\chi^P\boxtimes \chi^Q\big) = \int_{C\times C}\vp\big(\chi^P\boxtimes \chi^Q\big).
\label{rfbigf73hf983h8}
\end{align}
for some $\zeta\in H^0(\Oc_{C\times C}(3, 3, 0))$. Set $\gam = \zeta + \widetilde \zeta$. Since pairing on the left-hand side of \eqref{rfbigf73hf983h8} is defined through Serre duality it is well defined on cohomology. This means in particular, $\int_{C\times C}\gam(\chi^P\boxtimes \chi^Q) = \int_{C\times C}\gam(\chi^m\boxtimes \chi^Q) = \int_{C\times C}\gam(\chi^P\boxtimes \chi^n) = \int_{C\times C}\gam(\chi^m\boxtimes \chi^n)$ for any $\chi^m$ resp. $\chi^n$ equivalent to $\chi^P$ resp. $\chi^Q$ as in \eqref{rg87gf783h89fh309f390}. Hence,
\begin{align}
\int_{C\times C}\gam(\chi^P\boxtimes\chi^Q) = \int_{C\times C}\gam(\chi^m\boxtimes\chi^n)
\label{fbchjdvjebiuen}
\end{align}
for any $\chi^m, \chi^n$ satisfying \eqref{rg87gf783h89fh309f390}. Now looking at the right-hand side of \eqref{rfbigf73hf983h8} see that 
\begin{align}
\int_{C\times C}\vp(\chi^P\boxtimes \chi^Q)
&= 
\left\langle 
\vp - (\widetilde{res}~\vp)\cdot \mathbb S_{(C, K_C^{1/2})}, 
(\chi^P, \chi^Q, \chi^{PQ})\right\rangle
\notag
\\
&=
\left\langle 
\vp - (\widetilde{res}~\vp)\cdot \mathbb S_{(C, K_C^{1/2})}, 
(\chi^m, \chi^n, \chi^{mn})\right\rangle 
\label{fcbyevyuevbek}
\\
&= 
\int_{C\times C}\vp (\chi^m\boxtimes \chi^n)
\label{ruyefg3fiyvubvie}
\end{align}
where $(\chi^P, \chi^Q, \chi^{PQ})$ are deformation parameters for $\Xc \ra \Abb^{0|q}_k$ and $(\chi^P, \chi^Q, \chi^{PQ})$ and $(\chi^m, \chi^n, \chi^{mn})$ are gauge equivalent. Hence that \eqref{fcbyevyuevbek} follows from Proposition \ref{rf873f893hf3j0}. In putting together the equalities in \eqref{rfbigf73hf983h8}, \eqref{fbchjdvjebiuen} and \eqref{ruyefg3fiyvubvie} we can conclude that $\int_{C\times C} \vp\big(\chi^m\boxtimes \chi^n\big) = \int_{C\times C} \gam\big(\chi^m\boxtimes \chi^n\big)$. This holds for \emph{any} deformation parameters since
%, by Lemma \ref{rfh89hf983hf0309}, 
Schiffer variations at the associated points $P, Q$ will generate the odd, versal deformation space of the generic, even, genus $g = 2$ super Riemann surface $\Scl$. The proposition now follows.
\end{proof}

\begin{REM}\label{rbfef87f83f783}
\emph{It may be instructive to compare Proposition \ref{fjbvhfrbvrbvkjnk} with D'Hoker-Phong's gauge slice invariance of their gauge-fixed formula for contributions to the superstring measure, detailed in \cite{HOKERPHONG1}.}
\end{REM}

\begin{PROP}\label{rgf784gf874hf984}
At a generic, even, genus $g = 2$ spin curve $(C, K_C^{1/2})$ we have $h^0(\Oc_{C\times C}(3, 3, -1))^+ = 0$.
\end{PROP}

\begin{proof}
In Proposition \ref{fjbvhfrbvrbvkjnk} we have the relation $\int_{C\times C}\vp(\chi^m\boxtimes\chi^n) = \int_{C\times C}\gam (\chi^m\boxtimes \chi^n)$ for all odd, versal deformation parameters $(\chi^m, \chi^n)$. Recall that $\int_{C\times C}\gam(\chi^m\boxtimes \chi^n)$ is defined through the Serre duality pairing. 
By non-degeneracy of this pairing see that if $\int_{C\times C}\gam(\chi^m\boxtimes\chi^n) = \int_{C\times C}\gam^\p (\chi^m\boxtimes\chi^n)$ for all $\chi^m, \chi^n$, then necessarily $\gam = \gam^\p$. Now recall from the proof of Proposition \ref{fjbvhfrbvrbvkjnk} that $\gam = \vp - (\widetilde{res}~\vp)\cdot\mathbb S_{(C, K_C^{1/2})}$ for $\widetilde{res}~\vp$ a lift of $res~\vp$. Let $\widetilde{res}^\p~\vp$ be another lift of $res~\vp$ and set $\gam^\p = \vp - (\widetilde{res}^\p~\vp)\cdot \mathbb S_{(C, K_C^{1/2})}$. Arguing as in Proposition \ref{fjbvhfrbvrbvkjnk} we will find that $\int_{C\times C}\vp(\chi^m\boxtimes \chi^n) = \int_{C\times C} \gam^\p(\chi^m\boxtimes \chi^n)$. Hence that $\int_{C\times C}\gam(\chi^m\boxtimes \chi^n) = \int_{C\times C}\gam^\p(\chi^m\boxtimes \chi^n)$ for all $(\chi^m, \chi^n)$. The aforementioned non-degeneracy of Serre duality now demands $\gam =\gam^\p$ and therefore $(\widetilde{res}~\vp)\cdot \mathbb S_{(C, K_C^{1/2})} = (\widetilde{res}^\p~\vp)\cdot \mathbb S_{(C, K_C^{1/2})}$. By Lemma \ref{rfg784gf784hf9h383} this can only be true for any $\vp\in H^0(\Oc_{C\times C}(3, 3, 1))$ if $h^0(\Oc_{C\times C}(3, 3, -1))^+ = 0$. 
\end{proof}

\noindent
Through the proof of Proposition \ref{rgf784gf874hf984} we see, for $\vp\in H^0(\Oc_{C\times C}(3, 3, 1))^+$, that the image in $H^0(\Oc_{C\times C}(3, 3, 1))^+$ of any lift of $res~\vp$ will be independent of the particular choice of lift. Crucially then:

\begin{COR}\label{rbfuyvfyuf9h3f784}
A splitting of the sequence in Proposition \ref{rfg78gf78hf893} at a generic, even, genus $g= 2$ spin curve $(C, K_C^{1/2})$ is given by $\vp \mapsto \vp - \widetilde{res}~\vp\cdot \mathbb S_{(C, K_C^{1/2})}$ for $\vp\in H^0(\Oc_{C\times C}(3, 3, 1))$ and $\widetilde{res}~\vp\in H^0(\Oc_{C\times C}(2, 2, 0))$ a lift of $res~\vp$. 
 \qed
\end{COR}

\noindent
It remains now to globalize our argument to supermoduli space.

\subsubsection{Open sets in moduli spaces}
Our proof of Theorem \ref{rhf893h8f3h0fj39} reduces to showing a certain sequence of sheaves over the spin moduli space is split. This means we need to show there exists a covering of the moduli space with respect to which the restriction of this sequence to any open set in the covering will split. We briefly discuss open sets here in the context of moduli theory. For finer details, see \cite[\S4]{HARTDEF}, \cite[Ch. XIII]{ARABGRIFF}. In a moduli space $\Mcl$ over a field $k$, points of $\Mcl$ are morphisms $\mathrm{Spec}~k \ra \Mcl$ and open sets $\Uc$ correspond to families $X \ra \Uc$ of a prescribed kind. That is, to any family $X \stackrel{\pi}{\ra} \Uc$ there is a mapping $\Uc \ra \Mcl$ given by $u \mapsto (\mbox{isom. class of}~\pi^{-1}(u))$. If $\Cc\ra \Mcl$ is the universal family, then to any  $X \ra \Uc$ the total space $X$ can be realized as the cartesian pullback $X\cong \Uc \times_\Mcl\Cc$. In the case where $\Mcl = \Mcl_g$ is the moduli space of curves, $\Cc = \Mcl_{g, 1}$ is the moduli space of 1-pointed curves, realised as a family over $\Mcl_g$ by forgetting the marked point. For any $\Uc \subset \Mcl_g$ we have a commutative diagram
\begin{align*}
\xymatrix{
\ar[d]_\pi X \ar[r] & \Mcl_{g, 1} \ar[d]^p
\\
\Uc \ar[r] & \Mcl_g
}
&&
\mbox{given set theoretically by:}
&&
\xymatrix{
\ar@{|->}[d] (u, pt) \ar@{|->}[r] & [(\Xc_u, pt)]\ar@{|->}[d]
\\
u\ar@{|->}[r] & [\Xc_u]
}
\end{align*}
With $p : \Mcl_{g, 1} \ra \Mcl_g$ continuous, the pre-image of open sets under $p$ will be open. Furthermore, by universality of the cartesian pullback we have $X \cong p^{-1}\Uc$. Hence $X \subset \Mcl_{g, 1}$ will be open. For a sheaf $\Fc$ on $\Mcl_{g, 1}$ base-change theorems\footnote{See e.g., \cite[Ch. III, \S8]{HARTALG}} will imply that over any open set $\Uc$,
\begin{align}
(p_*\Fc)|_\Uc \cong \pi_*(\Fc|_X).
\label{rfh73hf893f8j380j09}
\end{align}
We intend on applying these conventions to the situation encapsulated by \eqref{rfh3hf98hf8030}.

\subsubsection{The relative Szeg\"o kernel}\label{rfh78gf74hf89h8f04}
Recall from \S\ref{rhf784f89hf09j39f03} that the space of Szeg\"o kernel differentials on a spin curve $(C, K_C^{1/2})$
is parametrised by $H^0(\Oc_{C\times C}(1, 1, 1))$. In the case where our spin curve is generic and even, Lemma \ref{rf8938f38jf03jf0} shows that restriction-to-diagonal gives an isomorphism on global sections,\footnote{note, we refer to the isomorphism $res$ in Lemma \ref{rf8938f38jf03jf0} by $\psi$ here} 
\begin{align}
\psi: H^0(\Oc_{C\times C}(1, 1, 1))\stackrel{\sim}{\ra} H^0(\Oc_C)\cong \Cbb. 
\label{rfhhf893f903f4f4f43}
\end{align}
Over a generic, even spin curve the pre-image of the generator $\psi^{-1}1$ was referred to as the `normalized' Szeg\"o kernel differential and denoted $\mathbb S_{(C, K_C^{1/2})}$. Now consider a family of generic, even spin curves over a base $\Uc$. This consists of a family of curves $X \stackrel{\pi_X}{\ra} \Uc$ together with a relative line bundle $L\ra X$ and an isomorphism $L\otimes L \cong K_{\pi_X}$, where $K_{\pi_X}$ is the relative cotangent sheaf on $X$. This means $L = (K_{C_b}^{1/2})_{b\in B}$ is a family of spin structures with $(\pi_*L)_b = H^0(K_{C_b}^{1/2})$, $C_b = \pi^{-1}b$. In imitation of \eqref{rfh3hf98hf8030} now, form the cartesian product $X\times_\Uc X$ to get a commutative diagram
\begin{align}
\xymatrix{
\ar[d]_{p_1} X\times_\Uc X \ar[rr]^{p_2} & & X\ar[d]^{\pi_X}
\\
X\ar[rr]^{\pi_X} & & \Uc.
}
\label{rfh894hf89f09j3}
\end{align}
As in \eqref{rfg8gf93h8390} then, set: 
\begin{align}
\widetilde\Oc_{X\times_\Uc X}(a, b, c)
\stackrel{\Delta}{=}
p_1^*L^{a/2}
\otimes p_2^*L^{b/2}(c\Delta_X)
\label{rfh83jf093jf09j30}
\end{align}
where $\Delta_X : X\subset X\times_\Uc X$ is the diagonal embedding. With the projection $\pi\pi_X = \pi_Xp_1 = \pi_Xp_2$
see that $\psi$ in \eqref{rfhhf893f903f4f4f43} gives an isomorphism 
$\psi_\Uc : \pi\pi_{X*}\widetilde\Oc_{X\times_\Uc X}(1, 1, 1)\stackrel{\sim}{\ra} ({\bf R}^{1}\pi_{X*}\om_{X/\Uc})^*$ for $\om_{X/\Uc}$ the relative, canonical bundle. Evidently, there exists a section $1_\Uc \in ({\bf R}^{1}\pi_{X*}\om_{X/\Uc})^*$ which restricts to the generator $1_u\in H^0(\Oc_{C_u})\cong \Cbb$ on each fiber $C_u\subset X$ over $u\in \Uc$.
%$\psi_\Uc : \pi\pi_{X*}\widetilde\Oc_{X\times_\Uc X}(1, 1, 1) \stackrel{\sim}{\ra} \Oc_\Uc$. 
Its pre-image $\psi_\Uc^{-1}1_\Uc$ is the \emph{relative Szeg\"o kernel differential} $\mathbb S_\Uc$ for the family of generic, even spin curves $X\stackrel{\pi_X}{\ra} \Uc$. In viewing $\psi_\Uc = (\psi_u)_{u\in \Uc}$ as a family of isomorphisms \eqref{rfhhf893f903f4f4f43} this differential over $\Uc$ is given by the assignments $u \mapsto \psi_u^{-1}1_u$.

\begin{LEM}\label{ruerivgiuhrouvjioe}
Over the locus of generic, even spin curves $\Scl\Mcl^+_g$, there exists a section $\mathbb S_g\in H^0(\pi\pi_*\widetilde\Oc(1, 1, 1)|_{\mathcal S\mathcal M_g^+})$ such that $\mathbb S_g|_\Uc = \mathbb S_\Uc$ for all open $\Uc\subset \Scl\Mcl_g^+$.
\end{LEM}

\begin{proof}
Starting from \eqref{rfh894hf89f09j3} with $\pi\pi_X : X\times_\Uc X \ra \Uc$ the projection we obtain a relation with \eqref{rfh3hf98hf8030} being,
\[
\xymatrix{
\ar[d]_{\pi\pi_X} X\times_\Uc X \ar[rr] & & \mathcal S\Mcl_{g, 1}\times_{\Scl\Mcl_g}\Scl\Mcl_{g, 1}\ar[d]_{pp}
\\
\Uc \ar[rr] & & \Scl\Mcl_g.
}
\]
Evidently, open sets in $\mathcal S\Mcl_{g, 1}\times_{\Scl\Mcl_g}\Scl\Mcl_{g, 1}$ are products of families $X\times_\Uc X$. With the general formula in \eqref{rfh73hf893f8j380j09} we have then:
\begin{align}
\big(
pp_*
\widetilde\Oc(a, b, c)\big)(\Uc)
&=
pp_*
\widetilde\Oc(a, b, c)|_\Uc
\notag
\\
&\cong
\pi\pi_{X*}\widetilde\Oc(a, b, c)|_{X\times_\Uc X}
&&\mbox{(from \eqref{rfh73hf893f8j380j09})}
\notag
\\
&=
\pi\pi_{X*}
\widetilde \Oc_{X\times_\Uc X}(a, b, c)
\label{rfh983hf983f09j3444}
\end{align}
where $\widetilde\Oc(a, b, c)$ is the sheaf from \eqref{rfg8gf93h8390} and $\widetilde \Oc_{X\times_\Uc X}(a, b, c)$ is as in \eqref{rfh83jf093jf09j30}. To give a section $\mathbb S_g\in H^0(pp_*\widetilde\Oc(1, 1, 1)|_{\Scl\Mcl_g^+})$ is equivalent to giving a section over each open set $\big(pp_*\widetilde\Oc(a, b, c)\big)(\Uc)$. Since we are working with the space parametrising generic, even spin curves this lemma follows from $\eqref{rfh983hf983f09j3444}|_{(a, b, c) = (1, 1, 1)}$ and the construction of the relative Szeg\"o kernel $\mathbb S_\Uc$.
\end{proof}

\subsubsection{Proof of Theorem \ref{rhf893h8f3h0fj39}: global splitting}
We are now in a position to complete the proof of Theorem \ref{rhf893h8f3h0fj39}. Recall, we want to show that the direct image of the sequence \eqref{rfj039jf093jf90jf3j} on $\Scl\Mcl_g|_{g=2}$ splits. Equivalently that the following sequence
\begin{align}
0
\lra
pp_*\widetilde \Oc(3, 3, 0)
\lra 
pp_*\widetilde \Oc(3, 3, 1)
\lra 
pp_*\Delta_*K_p^2
\lra
0.
\label{rfh74hf94hf04j0}
\end{align}
of sheaves on $\mathcal S\Mcl_g|_{g=2}$ splits. If we can argue this, Theorem \ref{rhf893h8f3h0fj39} will follow immediately from Theorem \ref{rhf983hf8030fj30} and Lemma \ref{rfu3f983hfj390f444}.\footnote{To elaborate, with $\dim \Mfr_g = (3g - 3|2g-2)$,\ see that $\Mfr_g|_{g = 2}$ will be $(3|2)$-dimensional. Hence we can apply Lemma \ref{rfu3f983hfj390f444} in genus $g= 2$.} Central to our argument is the construction of the relative Szeg\"o kernel $\mathbb S_g$ from Lemma \ref{ruerivgiuhrouvjioe}. With this note that Proposition \ref{rfh7f983hf98330} globalizes to moduli space giving the diagram
\begin{align}
\xymatrix{
\ar[d]_{\cdot\mathbb S_g} pp_* \widetilde\Oc(2, 2, 0) \ar[rr]^\sim & & \otimes ^2pp_*\Delta_*K_p\ar[d]
\\
pp_*\widetilde\Oc(3, 3, 1) \ar[rr]^{res} & & pp_*\Delta_*K_p^2.
}
\label{rhfu4hf984f0j49f04}
\end{align}
Accordingly, Proposition \ref{rhf8hf98hf030} also globalizes. Now let $\Uc\subset \Scl \Mcl_g|_{g = 2}$ be an open set intersecting $\Scl\Mcl_g^+|_{g=2}$. With a family $X \stackrel{\pi_X}{\ra} \mathcal U$ note that we have an isomorphism $\pi\pi_{X*}\widetilde\Oc_{X\times_\Uc X}(a, b, c)\otimes k(u) \cong H^0(\Oc_{C_u\times C_u}(a, b, c))$, where $k(u)$ is the residue field at $u\in \Uc$.\footnote{See e.g., \cite[Ch. III, \S9]{HARTALG} for a more general case.} With \eqref{rhfu4hf984f0j49f04}, the proof of the fiber-wise splitting in Corollary \ref{rbfuyvfyuf9h3f784} adapts to split the central extension $\pi\pi_{X*}\widetilde\Oc_{X\times_\Uc X}(3, 3, 1)\otimes k(u)$ for any $u\in \Uc$. Hence we obtain a splitting of $\pi\pi_{X*}\widetilde\Oc_{X\times_\Uc X}(3, 3, 1)$ which, by \eqref{rfh983hf983f09j3444}, gives a splitting of $(pp_*\widetilde\Oc(3, 3, 1))(\Uc)$. 

Theorem \ref{rhf893h8f3h0fj39} now follows.
\qed

\bibliographystyle{alpha}
\bibliography{Bibliography}

\begin{thebibliography}{ACG11}

\bibitem[ACG11]{ARABGRIFF}
E.~Arabello, M.~Cornalba, and P.~Griffiths.
\newblock {\em Geometry of algebraic curves}, volume~2 of {\em 268}.
\newblock Springer-Verlag, Berlin, 1 edition, 2011.

\bibitem[Bet16]{BETTPHD}
K.~Bettadapura.
\newblock {\em Obstruction Theory for Supermanifolds and Deformations of
  Superconformal Structures}.
\newblock PhD thesis, The Australian National University,
  \href{http://hdl.handle.net/1885/110239}{hdl.handle.net/1885/110239},
  December 2016.

\bibitem[Bet18]{BETTHIGHOBS}
K.~Bettadapura.
\newblock Higher obstructions of complex supermanifolds.
\newblock {\em SIGMA}, 14(094), 2018.

\bibitem[Bet19a]{BETTANAL}
K.~Bettadapura.
\newblock Analytic and algebraic deformations of super riemann surfaces.
\newblock Available at
  \href{https://arxiv.org/abs/1911.07118}{arXiv:1911.07118} [math.AG], 2019.

\bibitem[Bet19b]{BETTOBSTHICK}
K.~Bettadapura.
\newblock Obstructed {Thickenings} and {Supermanifolds}.
\newblock {\em J. Geom. and Phys.}, 139:25--49, 2019.

\bibitem[Bet19c]{BETTSRS}
K.~Bettadapura.
\newblock On the problem of splitting deformations of super {Riemann} surfaces.
\newblock {\em Lett. Math. Phys.}, 109(2):381--402, February 2019.

\bibitem[CR88]{RABCRANE}
L.~Crane and J.~Rabin.
\newblock Super {Riemann} surfaces: Uniformization and {Teichmueller} theory.
\newblock {\em Comm. Math. Phys.}, (113):601--623, 1988.

\bibitem[CV19]{CODVIV}
V.~Codogni and F.~Viviani.
\newblock {Moduli} and periods of supersymmetric curves.
\newblock {\em Adv. Th. and Math. Phys.}, (2):345--402, 2019.

\bibitem[DM99]{QFAS}
P.~Deligne and J.~W. Morgan.
\newblock {\em Quantum Fields and Strings: A course for Mathematicians},
  volume~1, chapter Notes on Supersymmetry (following Joseph Bernstein), pages
  41--97.
\newblock American Mathematical Society, Providence, 1999.

\bibitem[DP89]{CHIRALSPLT}
E.~D'Hoker and D.~Phong.
\newblock Conformal scalar fields and chiral splitting on super {Riemann}
  surfaces.
\newblock {\em Comm. Math. Phys.}, 3(125):469--513, 1989.

\bibitem[DP02]{HOKERPHONG1}
E.~D'Hoker and D.~Phong.
\newblock Lectures on two loop superstrings.
\newblock {\em Conf. Proc. C0208124}, pages 85--123, 2002.

\bibitem[DP15]{HOKDEFCPLX}
E.~D'Hoker and D.~Phong.
\newblock Higher order deformations of complex structures.
\newblock {\em SIGMA}, 11(047):14p, 2015.

\bibitem[DW14]{DW2}
R.~Donagi and E.~Witten.
\newblock Super {Atiyah} classes and obstructions to splitting of supermoduli
  space.
\newblock {\em Pure and applied mathematics quarterly}, 9(4), 2014.
\newblock Available at: \href{http://arxiv.org/abs/1404.6257}{arXiv:1404.6257}
  [hep-th].

\bibitem[DW15]{DW1}
R.~Donagi and E.~Witten.
\newblock Supermoduli space is not projected.
\newblock In {\em Proc. Symp. Pure Math.}, volume~90, pages 19--72, 2015.

\bibitem[Gid92]{GPUNCT}
S.~B. Giddings.
\newblock Punctures on super {Riemann} surfaces.
\newblock {\em Comm. Math. Phys.}, (143):355--370, 1992.

\bibitem[GS82]{GREENSCHWARZ}
M.~B. Green and J.~H. Schwarz.
\newblock Supersymmetrical string theories.
\newblock {\em Phys. Lett. B}, 109(444), 1982.

\bibitem[Har77]{HARTALG}
R.~Hartshorne.
\newblock {\em Algebraic Geometry}.
\newblock Springer, 1977.

\bibitem[Har10]{HARTDEF}
R.~Hartshorne.
\newblock {\em Deformation Theory}.
\newblock Springer, 2010.

\bibitem[HS66]{HAWLEYSCHIFF}
N.~S. Hawley and M.~Schiffer.
\newblock half-order differentials on {Riemann} surfaces.
\newblock {\em Acta Math.}, 115:199--236, 1966.

\bibitem[Man88]{YMAN}
Y.~Manin.
\newblock {\em Gauge Fields and Complex Geometry}.
\newblock Springer-Verlag, 1988.

\bibitem[Pen83]{PENKOV}
I.~B. Penkov.
\newblock D-modules on supermanifolds.
\newblock {\em Invent. Math.}, 71:501--12, 1983.

\bibitem[Wit15a]{WITTHOLOSTRING}
E.~Witten.
\newblock Notes on holomorphic string and superstring theory measures of low
  genus.
\newblock In P.~Feehan, J.~Song, B.~Weinkove, and R.~Wentworth, editors, {\em
  Analysis, Complex Geometry, and Mathematical Physics: In Honor of Duong H.
  Phong}, volume 644. Contemporary Mathematics, 2015.
\newblock available at: \href{http://arxiv.org/abs/1306.3621}{arXiv:1306.3621}
  [hep-th].

\bibitem[Wit15b]{WITTSPRD}
E.~Witten.
\newblock The super period matrix with {Ramond} punctures.
\newblock {\em J. Geom. and Phys.}, 92:210--239, June 2015.

\bibitem[Wit19a]{WITTRS}
E.~Witten.
\newblock Notes on super {Riemann} surfaces and their moduli.
\newblock {\em Pure and applied mathematics quarterly}, 15(1):57--211, 2019.

\bibitem[Wit19b]{WITTSMFLD}
E.~Witten.
\newblock Notes on supermanifolds and integration.
\newblock {\em Pure and applied mathematics quarterly}, 15(1):3--56, 2019.

\end{thebibliography}

\hfill
\\
\noindent
\small
\textsc{
Kowshik Bettadapura 
\\
\emph{School of Mathematics and Statistics} 
\\
University of Melbourne
\\
Victoria, 3010, Australia}
\\
\emph{E-mail address:} \href{mailto:k.bettad@gmail.com}{k.bettad@gmail.com}

\end{document}